\documentclass[journal]{IEEEtran}
\usepackage{setspace}
\usepackage{graphics} 
\usepackage{epsfig} 
\usepackage{mathptmx} 
\usepackage{times} 
\usepackage{amsmath} 
\usepackage{amssymb}  
\usepackage{amsthm}
\usepackage{epstopdf}
\usepackage{cite}
\usepackage[final]{changes}
\usepackage{xcolor}
\usepackage{bbm}
\usepackage{tikz}
\usetikzlibrary{arrows}
\usepackage{caption}
\usepackage{subcaption}
\usepackage{algpseudocode}
\usepackage{mathtools}
\usepackage{algorithm}
\DeclarePairedDelimiter{\ceil}{\lceil}{\rceil}

\newtheorem{theorem}{Theorem}[section]
\newtheorem{lemma}[theorem]{Lemma}
\newtheorem{proposition}[theorem]{Proposition}
\newtheorem{problem}[theorem]{Problem}

\newtheorem{definition}[theorem]{Definition}
\newtheorem{assumption}[theorem]{Assumption}
\newtheorem{remark}[theorem]{Remark}

\newtheorem{example}[theorem]{Example}
\ifCLASSINFOpdf

\else

\fi

\begin{document}
	
	\title{Co-design of Safe and Efficient Networked Control Systems in Factory Automation \replaced{with State-dependent Wireless Fading Channels}{: A Constrained Cooperative Game Approach}}
	
	\author{Bin Hu, Yebin Wang, Philip~Orlik, Toshiaki~Koike-Akino and Jianlin~Guo
	
	\thanks{Yebin Wang, Philip Orlik, Toshiaki Koike-Akino and Jianlin Guo are with Mitsubishi Electric Research Laboratories~(MERL), Cambridge, MA 02139, USA. {\tt\small yebinwang, porlik, koike,  guo@merl.com}}
	\thanks{This work was performed during Bin Hu's internship at MERL. {\tt\small bhu@odu.edu}}}
	
	\maketitle
	
	\begin{abstract}
In factory automation, heterogeneous manufacturing processes need to be coordinated over wireless networks to achieve safety and efficiency. These wireless networks, however, are inherently unreliable due to \emph{shadow fading} induced by the physical motion of the machinery. To assure both safety and efficiency, this paper proposes a state-dependent channel model that captures the interaction between the physical and communication systems. By adopting this channel model, sufficient conditions on the maximum allowable transmission interval are then derived to ensure \emph{stochastic safety} for a nonlinear physical system controlled over a state-dependent wireless fading channel. Under these sufficient conditions, the safety and efficiency co-design problem is formulated as a constrained cooperative game, whose equilibria represent optimal control and transmission power policies that minimize a discounted joint-cost in an infinite horizon. This paper shows that the equilibria of the constrained game are solutions to a non-convex generalized geometric program, which are approximated by solving two convex programs. The optimality gap is quantified as a function of the size of the approximation region in convex programs, and asymptotically converges to zero by adopting a branch-bound algorithm. Simulation results of a networked robotic arm and a forklift truck are presented to verify the proposed co-design method.
	\end{abstract}


\begin{IEEEkeywords}
Co-design method, shadow fading, stochastic safety, factory automation, networked control system.
\end{IEEEkeywords}
	
	\IEEEpeerreviewmaketitle
	
	\section{Introduction}
\subsection{Background and Motivation}
\IEEEPARstart{F}{actory} Automation Networks (FANs) are Cyber-Physical Systems (CPS)  consisting of numerous heterogeneous manufacturing processes that coordinate with each other by exchanging information over wireless networks \cite{zhuang2007wireless, rajhans2014supporting,de2006use}. FANs have received considerable attention due to the rapid development of wireless communication technologies, which provides efficient and cost-effective service such as increased mobility, easy scalability and maintenance for applications like automated assembly systems in manufacturing factories \cite{groover2007automation}. In many safety-critical applications, safety is always of primary concern in FANs. However, building safe and efficient FANs is challenging in two aspects. First, from a system modeling standpoint, the heterogeneous nature of FANs requires a hybrid framework that can capture system dynamics in different levels as well as their mutual interactions. Assessing the performance and safety of this ``hybrid'' system as a whole demands different modeling and analysis tools. Secondly, the wireless network in FANs is inherently unreliable due to channel fading \cite{islam2012wireless, de2006use} or interference \cite{tse2005fundamentals} caused by internal system states or external environments, such as obstacles or physical motions of machinery. The fading channel inevitably results in a severe drop in the network's quality of service (QoS) and thereby introduces a great deal of stochastic uncertainties in FANs that may cause serious safety issues. The objective of this paper is to develop a co-design paradigm for communication and control systems under which a certain level of safety and efficiency can be achieved for FANs in the presence of \emph{shadow fading}.

Assuring safety for FANs often requires joint coordination from heterogeneous systems which may have different objectives. Such a coordination is necessary due to the interactions among the heterogeneous systems. \added{Such interactions exist in many industrial applications, to name a few,  manufacturing systems with heavy facilities mills and cranes discussed in \cite{agrawal2014long}, sensor network with moving robots \cite{quevedo2013state} and indoor wireless networks with moving human bodies \cite{kashiwagi2010time}.} One typical example in FANs \deleted{shown in Figure \ref{fig: hsf}} is an assembly process where an autonomous assembly arm and a forklift truck collaborate to assemble products. On the one hand, the control objective of an autonomous assembly arm is to track a specified trajectory by exchanging information between a physical plant and a remote controller via wireless networks. On the other hand, the objective of the forklift system is often related to accomplishing some high-level tasks, such as transporting assembled products from one workstation to another. These physically separated systems, however, may have strong cyber-physical couplings. The cyber-physical couplings in the systems of networked assembly arm and forklift trucks comes from the fact that the physical motion of forklift vehicle may lead to serious \emph{shadow fading} in the wireless network that is used by the assembly arm, thereby significantly affecting the system stability and performance. Thus, to ensure system safety for FANs, one must explicitly examine such cyber-physical couplings in communication channels. 

The channel model that is used to characterize the \emph{shadow fading} in FANs, must be carefully examined. As a type of channel fading, shadow fading is often characterized in terms of the channel gain.  Traditionally, the channel gains are modeled either as \emph{independent identical distributed}~(i.i.d.) random processes \cite{tse2005fundamentals, gatsis2014optimal, tatikonda2004control, elia2005remote} with assumed distributions such as Rayleigh, Rician and Weibull or as Markov chains \cite{zhang1999finite, wang1995finite}. These channel models are inadequate to characterize the cyber-physical couplings in FANs due to the fact that the network state is assumed to be independent from physical states in either i.i.d. or Markov chain models. With such independency, control and communication could be considered separately through the application of a separation principle \cite{gatsis2014optimal}. This separation-principle, may be valid for networked system where the network states are independent of physical dynamics, but is clearly inappropriate for FANs where the channel state is functionally dependent on the physical states. This dependency of channel states on physical states motivates the development of a new co-design paradigm under which the communication and control policies are coordinated to achieve both system safety and efficiency.  
\subsection{Related Work}
The example of an assembly process \replaced{as well as the research work in \cite{agrawal2014long, quevedo2013state,kashiwagi2010time,agrawal2009correlated,leong2016network} have}{shown in Figure \ref{fig: hsf} has} demonstrated the importance of considering the cyber-physical couplings between communication and control systems in assuring system safety and efficiency for FANs. Similar conclusions have also \replaced{been}{be} made in prior work \cite{hu2013using, hu2015distributed} where the dependency of channel states on physical states is used in the design of distributed switching control strategy to assure vehicle safety in vehicular networked systems. This paper expands the results in \cite{hu2015distributed} to show that both system safety and efficiency can be achieved via a novel co-design framework. Other than these papers, we are aware of no other work formally analyzing both the system safety and efficiency \emph{in the presence of  such cyber-physical couplings}. There is, however, a great deal of related work on the co-design of communication and control systems assuming the channel states are independent of physical states. We will review these results and discuss their relationships to the work in this paper. 

From a communication perspective, the impact of channel fading on the system performance can be mitigated by increasing the transmission power. This observation motivates much research on the design of optimal power strategy to achieve various objectives in both communication \cite{caire1999optimum, goldsmith1997variable} and control communities \cite{gatsis2014optimal, quevedo2013state, quevedo2014power}. The objective of power control in the communication community mainly focuses on improving the communication reliability and performance in an average or asymptotic sense. In \cite{caire1999optimum, goldsmith1997variable} and relevant references therein, an adaptive power strategy combined with adaptive data-rate strategies was developed to achieve Shannon limit for fading channels. The optimal power strategy was shown to be a function of the channel gain. 

The objective of power control in the control community, however, is more concerned with how the communication quality affects the system stability and performance.  As shown in \cite{tatikondacontrol, nair2007feedback}, such impact is often related to the unstable modes of the dynamics in physical systems and the QoS that could be delivered by a given wireless network. The power control strategy in networked control systems is often designed to ensure a certain level of QoS under which the closed-loop system is stable. In \cite{quevedo2013state, quevedo2014power}, sufficient conditions on the transmission power were established to ensure exponentially bounded performance for state estimation of discrete linear time-varying systems. 

When considering a joint objective for the communication and control systems, recent work in \cite{molin2009lqg, gatsis2014optimal, di2015co, bao2011iterative} showed that the \emph{certainty equivalence property} holds for the optimal control policy while the optimal communication policy was adapted to the channel states and physical states. In particular, \cite{molin2009lqg} showed that the joint optimization of scheduling and control can be separated into the subproblems of an optimal regulator, estimator and scheduling. Similar ideas were applied to a joint design of controller and routing redundancy over a wireless network \cite{di2015co}. The work in \cite{gatsis2014optimal} considered a co-design problem for optimal control and transmission power policies for a stochastic discrete linear system controlled over a fading channel. Their results showed that the optimal control policy was a standard LQR controller while the optimal power policy was adapted to both channel and plant states. This similar structure was also discovered in a joint design problem for an optimal encoder and controller over noisy channels \cite{bao2011iterative}.

All of the above studies, however, were developed by assuming a state-independent channel model. From a safety standpoint, this state-independent channel model is often obtained by assuming the worst impact that the physical state can have on the network. As a result, the selected communication policy~(transmission power, data rate, or scheduling) may be greater than necessary to assure the same level of performance that can be obtained by using state-dependent channel model. In other words, the conservativeness on the selection of state-independent channel model may prevent the system as a whole from achieving system efficiency.
\subsection{Contribution}
Motivated by the cyber-physical couplings in heterogeneous industrial systems, this paper develops a co-design paradigm to achieve both system safety and efficiency in the presence of \emph{shadow fading}. The heterogeneous industrial systems are characterized by a nonlinear networked control system and a Markov decision process, which can represent a variety of realistic situations in industrial applications \cite{agrawal2014long, quevedo2013state,kashiwagi2010time,agrawal2009correlated,leong2016network}. Under this heterogeneous system framework, the first contribution of this paper is the proposal of a novel state-dependent fading channel model that captures the impact of the physical states on the channel state. \added{Furthermore, this paper shows that the state-dependent channel model is a Markov modulated Bernoulli process \cite{ozekici1997markov} that generalizes the traditional i.i.d. Bernoulli channel model in two important aspects: (1) the model parameters are not constants and are stochastic processes due to their dependence on a randomly changing environment; (2) the channel parameters can be controlled by taking advantage of the cyber-physical couplings between communication and control systems. }

Under the state-dependent channel model, the safety issue is examined in a stochastic setting by investigating the likelihood of the system states entering a forbidden or unsafe region. Thus, the second contribution of this paper is the sufficient condition on the maximum allowable transmission interval (MATI) under which the wireless networked system with \emph{state-dependent fading channels} is \emph{stochastically safe}. \added{We also show that the MATI derived in this paper generalizes the well known results in \cite{nevsic2004input} where the channel fading impact was not considered.  To the best of our knowledge, the sufficient
conditions presented in this paper are the first results on MATI that guarantee the stochastic safety under the \emph{state-dependent fading channels}.}

Under these safety conditions, the third contribution of this paper is the proposal of a new co-design paradigm to assure both safety and efficiency for FANs. In particular, we show that this safety-efficiency co-design can be formulated as a constrained two-player cooperative game. The equilibrium points of the constrained cooperative game represent optimal control and transmission power policies that minimize a discounted joint-cost induced by power consumption and control efforts in infinite horizon. The equilibrium of this constrained cooperative game can be obtained by solving a non-convex generalized geometric program (GGP) \cite{boyd2007tutorial, maranas1997global}. To address the non-convexity of the GGP, this paper approximates the non-convex GGP with two relaxed convex GGPs that provide upper and lower bounds on the optimal solution. These bounds are shown to asymptotically approach the global optimum by using a branch-bound algorithm. 

This paper is organized as follows. Section \ref{section: SHS} describes the system model and problem formulation. Section \ref{sec: safety} presents the sufficient conditions to ensure \emph{stochastic safety}. Under the safety conditions, Section \ref{sec: efficiency} proposes a co-design paradigm to assure both safety and efficiency. The optimal solutions for the co-design problem are provided in Section \ref{sec: GGP}. The main results are demonstrated via simulations of a mechanical robotic arm and a forklift truck in Sections \ref{sec: example}. Section \ref{sec: conclusion} concludes the paper. 

\textbf{Notations}. Throughout the paper the $n$-dimensional Euclidean vector space is denoted by $\mathbb{R}^{n}$ and the non-negative reals and integers are denoted as $\mathbb{R}_{\geq 0}$ and $\mathbb{Z}_{\geq 0}$, respectively. The infinity norms of the vector $x \in \mathbb{R}^{n}$ and the matrix $A$ are denoted by $|x|$ and $\|A\|$ respectively. The right limit value of a function $f(t)$ at time $t$ is denoted by $f(t^{+})$. Given a time interval $[t_1, t_2)$ with $t_1, t_2 > 0$, the essential supremum of a function $f(t)$ over the time interval $[t_1, t_2)$ is denoted by $|f(t)|_{[t_1, t_2)}=\text{ess}\sup_{t \in [t_1, t_2)}\|f(t)\|$ where $\|\cdot\|$ is the Euclidean norm. A function $f(t)$ is essentially ultimately bounded if $\exists M > 0$, $|f(t)|_{\mathcal{L}_{\infty}}=\text{ess}\sup_{t \geq 0}\|f(t)\| \leq M$. A function $\alpha(\cdot): \mathbb{R}_{\geq 0} \rightarrow \mathbb{R}_{\geq 0}$ is a class $\mathcal{K}$ function if it is continuous and strictly increasing, and $\alpha(0)=0$. A function $\alpha(t)$ is a class $\mathcal{K}_{\infty}$ function if it is in class $\mathcal{K}$ and radially unbounded. A function $\beta(\cdot, \cdot): \mathbb{R}_{\geq 0} \times \mathbb{R}_{\geq 0} \rightarrow \mathbb{R}_{\geq 0}$ is a class $\mathcal{KL}$ function if $\beta(\cdot, t)$ is a class $\mathcal{K}_{\infty}$ function for each fixed $t \in \mathbb{R}_{\geq 0}$ and $\beta(s, t) \rightarrow 0$ for each $s \in \mathbb{R}_{\geq 0}$ as $t \rightarrow +\infty$. \added{The function $\beta(\cdot, \cdot)$ is said to be of class Exp-$\mathcal{KL}$ if there exist $K_1, K_2 > 0$ such that $\beta(s, t)=K_1 \exp(-K_{2}t)s$. A function $\overline{\beta}(\cdot, \cdot, \cdot): \mathbb{R}_{\geq 0} \times \mathbb{R}_{\geq 0} \times \mathbb{R}_{\geq 0} \rightarrow \mathbb{R}_{\geq 0}$ is said to be of class $\mathcal{KLL}$~($\overline{\beta} \in \mathcal{KLL}$), if for each $r \geq 0$, $\overline{\beta}(\cdot, \cdot, r) \in \mathcal{KL}$ and $\overline{\beta}(\cdot, r, \cdot) \in \mathcal{KL}$.}
	\section{System Model: A Heterogeneous System Framework}
\label{section: SHS}
\added{Fig. \ref{fig: hsf} shows a heterogeneous system framework with two subsystems.  One is a networked control system ($\mathcal{G}$) that characterizes a nonlinear physical system being controlled over a wireless network. The other one is a Markov Decision Process (MDP) ($\mathcal{M}$) that models stochastic high level dynamics of a moving object in industrial systems. }
	
\added{The cyber-physical coupling within this heterogeneous framework is due to the fact that the physical states~(e.g., locations) of the moving object modeled by MDP's states may lead to \emph{shadow fading} on the wireless channel that is used by the networked control system. Such a coupling has been shown to be critical for performance guarantee in a variety of realistic situations in industrial applications, to name a few, such as robotic arms and forklift trucks, heavy facilities mills and cranes \cite{agrawal2014long}, sensor network with moving robots \cite{quevedo2013state,leong2016network} and indoor wireless networks with moving human bodies \cite{kashiwagi2010time}. Under such industrial settings, the radio channel characteristics are non-stationary and may experience abrupt changes due to the motion of the moving object. Such \emph{state-dependent} property of these wireless communications in industrial systems clearly invalidates the use of traditional co-design frameworks, such as \cite{molin2009lqg,gatsis2014optimal, zhang2006communication}, that rely on the assumption that the channel states are decoupled from the physical states. The heterogeneous system framework depicted in Fig. \ref{fig: hsf} is thus motivated by the co-design challenge under state-dependent fading channels.}

\deleted{Figure \ref{fig: hsf} shows a heterogeneous system framework with two subsystems. One is a networked control system ($\mathcal{G}$) that characterizes a physical robotic manipulator being controlled over a wireless network. The other one is a Markov Decision Process (MDP) ($\mathcal{M}$) that models the high transitions of a forklift truck. The MDP's states may lead to \emph{shadow fading} on the wireless channel that is used by the networked control system. This cyber-physical coupling is modeled by a state dependent channel model in Section \ref{sub: channel-model}.}
\begin{figure}[!t]
	\centering
	\includegraphics[scale=.25]{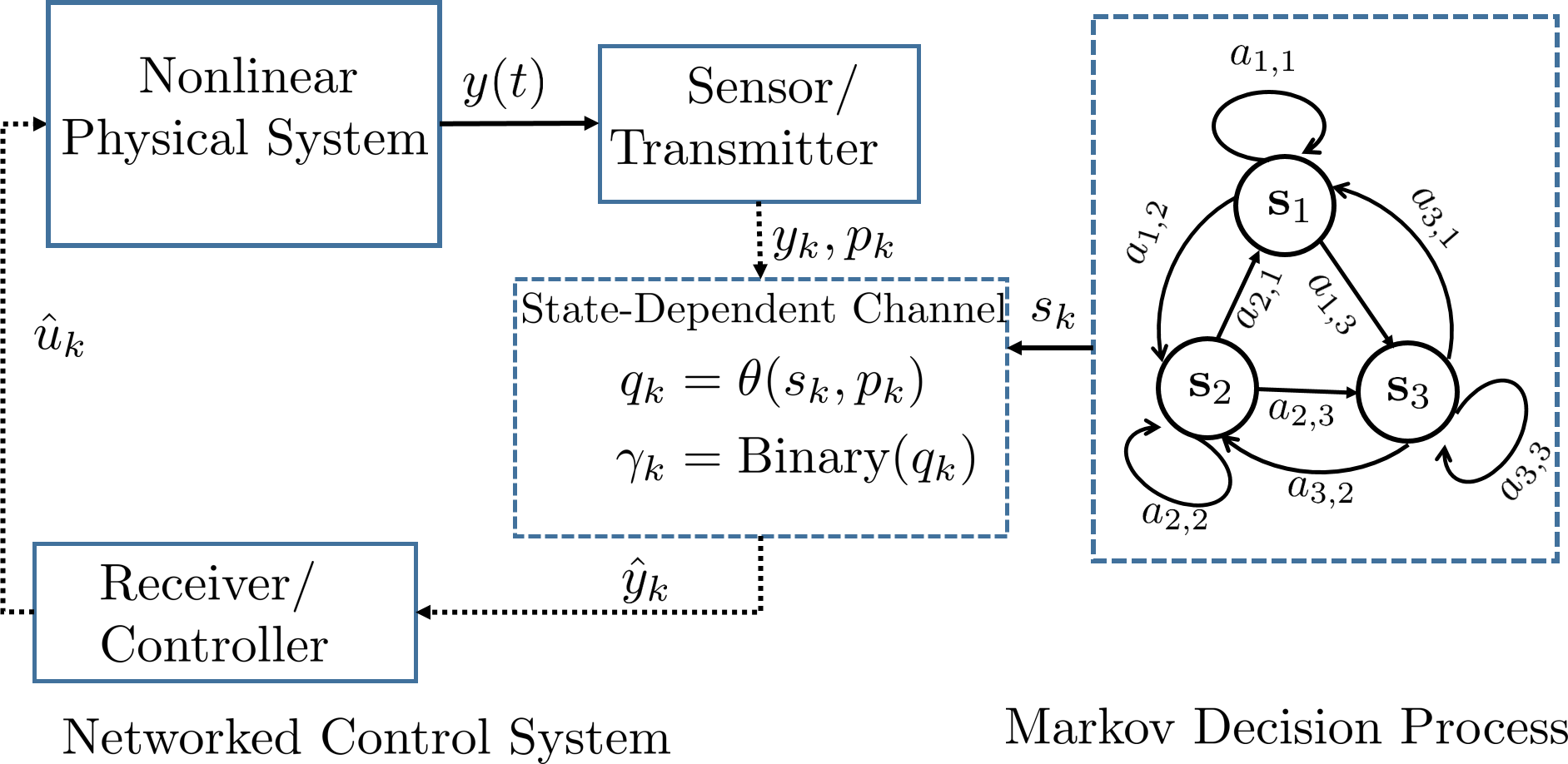}
	\caption{Heterogeneous System Framework: Networked Control System and Markov Decision Process}
	\label{fig: hsf}
\end{figure}
 \subsection{The $\mathcal{G}$ System Model}
The dynamics of the ${\mathcal G}$ system are modeled as follows,
\begin{align*}
\mathcal{G}:=\begin{cases} \dot{x}_{p}&=f_{p}(t, x_p, \hat{u}, w) \\
                                              y&=g_{p}(t, x_p)    , \quad \text{Physical Plant} \\
                                              \dot{x}_{c}&=f_{c}(t, x_{c}, \hat{y}) \\
                                              u&=g_{c}(t, x_{c}),  \quad \text{Remote Controller}. \\
\end{cases}
\end{align*}
where $x_{p} \in \mathbb{R}^{n_{x}}$ and $y \in \mathbb{R}^{n_{y}}$ are the physical states and measurements, respectively. $x_{c} \in \mathbb{R}^{n_{c}}$ and $u \in \mathbb{R}^{n_{u}}$ are the internal state and output for the remote controller, respectively. $w \in \mathbb{R}^{n_w}$ is the external disturbance that is assumed to be essentially ultimately bounded, i.e., $\exists M_w > 0$, $|w|_{\mathcal{L}_{\infty}} \leq M_w$. $f_{p}(\cdot, \cdot, \cdot, \cdot):\mathbb{R}_{\geq 0}\times \mathbb{R}^{n_{x}} \times \mathbb{R}^{n_{u}} \times \mathbb{R}^{n_w} \rightarrow \mathbb{R}^{n_x}$, $g_{p}(\cdot, \cdot): \mathbb{R}_{\geq 0} \times \mathbb{R}^{n_x} \rightarrow \mathbb{R}^{n_y}$, $f_{c}(\cdot, \cdot, \cdot):\mathbb{R}_{\geq 0}\times \mathbb{R}^{n_c} \times \mathbb{R}^{n_{u}} \rightarrow \mathbb{R}^{n_c}$ and $g_{c}(\cdot, \cdot): \mathbb{R}_{\geq 0} \times \mathbb{R}^{n_x} \rightarrow \mathbb{R}^{n_u}$ are Lipschitz functions for the physical plant and remote controller respectively.  Without loss of generality, we assume the origin is the unique equilibrium for system $\mathcal{G}$, i.e. $f_{p}(0, 0, 0, 0)=0^{n_x}, f_{c}(0, 0, 0)=0^{n_{c}}, g_{p}(0, 0)=0^{n_y}, g_{c}(0, 0)=0^{n_u}$.

Let $\{t_k\}$ denote an increasing sequence of
time instants where $t_k < t_{k+1}$ for all $k \in \mathbb{Z}_{\geq 0}$.  Let $\Omega_{p}=\{p_{i}\}_{i=1}^{M}$ be a transmission power set including $M$ power levels where $p_{i} \in \mathbb{R}_{\geq 0}$ is the power level. As shown in Figure \ref{fig: hsf}, the measurement $y$ and controller output $u$ are sampled and transmitted over an unreliable
communication channel with a selected power level $p_k \in \Omega_{p}$ at time instant $t_k$. The wireless network is subject to fading and randomly drops the sampled information at each time instant. Let $\{\gamma(k)\}$ denote a binary random process taking value from $\{0, 1\}$.  The value of the process $\gamma(k)$ at the $k$th consecutive sampling instant indicates whether or not a packet dropout has occurred. In particular,
\begin{align*}
\gamma(k)=\begin{cases} 1 & ,\quad \text{packet successfully decoded without error} \\
0 & , \quad \text{packet is dropped}. \\
\end{cases}
\end{align*}
Let $\hat{y}(t_{k})$ and $\hat{u}(t_{k})$ denote the estimates of the corresponding variables at time instant $t_{k}$. Note that we assume the time used for communication and computing control action is negligible compared to the sampling time interval and the network condition is unchanged during this small time interval. The estimation error induced by the communication during the sampling time interval $[t_{k}, t_{k+1})$ is defined as $e_{y}(t)=y(t)-\hat{y}(t_{k})$ and $e_{u}=u(t)-\hat{u}(t_{k})$. Let $e(t) = [ e_y(t); e_u(t)]^T$ denote the aggregated estimation error at time $t$. After
the information is successfully received, this aggregated estimation error will be reset to zero.  Let $t_k^+$ denote
the real time immediately after the sampling instant, $t_k$. The estimation error $e(t_k^+)$
will be reset to zero immediately after each successful transmission.  So we may formally express $e(t_k^+)$ as $e(t_{k}^{+})=(1-\gamma(k))e(t_{k})$.
Let $x:=[x_{p}; x_{c}]$ denote the aggregated state for the closed loop system $\mathcal{G}$, and then one has the following equivalent system representation in terms of $x$ and $e$,
\begin{align}
\hat{\mathcal{G}}:=\begin{cases}
\dot{x}&=f(t, x, e, w) \\
\dot{e}&=g(t,x, e, w), \forall t \in (t_{k}, t_{k+1})  \\
e(t_{k}^{+})&=(1-\gamma(k))e(t_{k}), \quad k \in \mathbb{N}^{+}.\\
\end{cases}
\label{eq: G_hat}
\end{align}
where 
\begin{align*}
f(t, x, e, w)&:=\left [\begin{array}{c}   f_{p}(t, x_{p}, g_{c}(t, x_{c})-e_{u}(t), w)  \\
                                                       f_{c}(t, x_{c}, g_{p}(t, x_p)-e_{p}(t))
 \end{array} \right] \\
 g(t, x, e, w)&:=\left [\begin{array}{c}      \frac{\partial g_{p}(x_p, t)}{\partial x_{p}}f_{p} (t, x_{p}, g_{c}(t, x_{c})-e_{u}(t), w) + \frac{\partial g_{p}(x_p, t)}{\partial t}                \\
                                                             \frac{\partial g_{c}(x_c, t)}{\partial x_{c}}f_{c}(t, x_{c}, g_{p}(t, x_{p})-e_{u}(t))+\frac{\partial g_{c}(x_c, t)}{\partial t}
 \end{array} \right].
\end{align*}
Note that  we further assume that the functions $g_{p}(\cdot, \cdot)$ and $g_{c}(\cdot, \cdot)$ are continuously differentiable and thus the function $g(\cdot, \cdot, \cdot, \cdot)$ in \eqref{eq: G_hat} is well defined. Since the (set)~stability of the system $\hat{\mathcal{G}}$ implies the (set)~stability of the system $\mathcal{G}$, we will only discuss the stability of the system $\hat{\mathcal{G}}$ in the remaining of this paper.  
\subsection{The $\mathcal{M}$ System Model}
The $\mathcal{M}$ system is modeled by an MDP process. An MDP is defined by a five tuple $\mathcal{M}=\{S, s_{0}, A, P, c\}$, where
\begin{itemize}
\item $S=\{s_{i}\}_{i=1}^{N}$ is the state space for the MDP.
\item $s_{0} \subset S$ is the set of initial states.
\item $A=\{a_{i}\}_{i=1}^{M_{a}}$ is the action set.
\item $P: S \times A \times S \rightarrow [0, 1]$ is the transition probability , i.e. $P(s_{i}, a, s_{j})={\rm Pr}\{s_{j}| a, s_{i}\}$.
\item $c: S \times A \rightarrow \mathbb{R}_{\geq 0}$ is the reward function. 
\end{itemize}

Unlike system $\mathcal{G}$ that models low level physical dynamics, the MDP process is used to model discrete-event decision making
processes managing high-level control objectives such as transporting products from one location to another with minimum time or energy. \added{The state space $S$ in the MDP system corresponds to a finite number of partitioned regions that the vehicle system, such as forklift trucks or cranes \cite{agrawal2014long} or robots \cite{quevedo2013state}, can operate by taking actions from an action set $A$.  The transition probability matrix $P$ is used to model the stochastic uncertainties caused by sensor or actuation noises when the actions are physically implemented. The costs in the MDP model are defined to characterize the high level control objectives for the vehicle system. For instance, if the control objective is to transport the products to a target region, then small costs will be assigned in the minimization optimization problem, to the situation when the vehicle is transitted to the target region.}
\subsection{State Dependent Dropout Channel Model}
\label{sub: channel-model}
As shown in Fig. \ref{fig: hsf}, the wireless channel used by the networked control system $\mathcal{G}$ is functionally dependent on the state of the MDP system. This relationship corresponds to the situation that vehicle's physical positions directly lead to \deleted{the} shadow fading, thereby generating a great deal of stochastic uncertainties in system $\mathcal{G}$. Equation (\ref{eq: G_hat}) shows that the stochastic uncertainty in system $\hat{\mathcal{G}}$ is governed by a binary random process $\{\gamma(k)\}$, which characterizes the stochastic variations in channel conditions. 

The state-dependency in the shadow fading channel is captured by a novel \emph{State-Dependent Dropout Channel} (SDDC) model that is formally defined as follows.
\begin{definition}
\label{def: SDDC}
Given a binary random process $\{\gamma(k)\}_{k=0}^{\infty}$, an MDP system $\mathcal{M}=\{S, s_{0}, A, P_{m}, c\}$ and a transmission power set $\Omega_{p}=\{p_{i}\}_{i=1}^{M}$,  the wireless channel is SDDC if 
\begin{align}
\label{eq: SDDC}
{\rm Pr}\{\gamma(k)=1 | s(k)=s, p(k)=p\} = 1-\theta(s, p), \forall s \in S, p \in \Omega_p.
\end{align}
where $\theta(s, p) \in (0, 1)$ is the outage probability \cite{tse2005fundamentals} that monotonically decreases with respect to the transmission power level $p$.
\end{definition}
\begin{remark}
	\added{
The definition of the SDDC is closely related to  the \emph{outage probability}, which is a widely used performance metric for fading channels \cite{tse2005fundamentals}. It characterizes the likelihood of the Signal-to-Noise Ratio~(SNR) \added{being} below a specified threshold $\gamma_0$, i.e. ${\rm Pr}\{\text{SNR} \leq \gamma_0\}$. The difference between the SDDC model and traditional \emph{outage probability} lies in the state-dependent feature of \eqref{eq: SDDC} where the probability is defined for each each MDP state~(partitioned region). The probability defined in \eqref{eq: SDDC} can be obtained by measuring the SNR for each MDP state, see \cite{kashiwagi2010time, agrawal2014long} and reference therein for details about the statistical methods. In practice, the transmitter can estimate the probability by either directly using the visual sensor to observe the positions of the controlled moving object, or using the estimation techniques discussed in \cite{agrawal2014long,agrawal2009correlated}. See Example \ref{example: raleigh-fading} for more details about how to construct the SDDC from the \emph{outage probability}.}
\end{remark}

\begin{remark}
	\added{
The SDDC model in \eqref{eq: SDDC} relates the channel state~(packet dropout probability) to the MDP states and transmission power levels. From a control standpoint, this correlation enables that the channel conditions can be controlled by designing different control and transmission power strategies. By using such a freedom in the channel model, this paper develops a co-design framework that coordinates control and communication strategies to achieve both safety and efficiency for the entire heterogeneous system. The co-design idea of using the state-dependent channel model distinguishes our work from other results, such as \cite{quevedo2013state, kashiwagi2010time, agrawal2014long, leong2016network} where the channel state is assumed a fixed and uncontrollable random process.}
\end{remark}

\begin{example}[SDDC model with Raleigh fading]
\label{example: raleigh-fading}
Channel fading is often the result of the superimposition of signal attenuation in both large~(shadowing) and small scale levels \cite{tse2005fundamentals}. Let $h_k$ denote the small scale fading gain induced by multi-path propagation at time instant $t_k$.  Suppose $\{h_{k}\}_{k=0}^{\infty}$ is an \emph{i.i.d} process that
satisfies a Raleigh distribution with a scale parameter $1$, i.e. $h_{k} \sim \text{Raleigh}(1), \forall k \in \mathbb{Z}_{\geq 0}$. Let $\psi(\cdot): S \rightarrow [0, 1]$ denote a shadow level function that characterizes the level of shadowing effect on the channel gain for each MDP state, i.e. $0 \leq \psi(s) \leq 1, \forall s \in S$. Thus, the state dependent channel gain is $\overline{h}_{k}(s):=\psi(s)h_{k}$, and for a given transmission power level $p$ and noise power $N_0$, the SNR is $p\overline{h}_{k}(s)^{2}/N_0$. With the assumption that the small scale fading gain is conditionally independent on shadowing state $s \in S$,  for a given SNR threshold $\gamma_0$, one has
\begin{align*}
&{\rm Pr}\{\gamma(k)=1 | s(k)=s,p(k)=p\}\\=&{\rm Pr}\{\frac{p(k)h_{k}^{2}\psi(s(k))^{2}}{N_{0}} \geq \gamma_{0} \Big| s(k)=s, p(k)=p\} \\                                                             =&\int_{\frac{\gamma_{0}N_{0}}{p}}^{\infty} \psi(s) e^{-\psi(s) x} dx = e^{-\frac{N_{0}\gamma_{0}\psi(s)}{p}}.
\end{align*}
Then, we have the explicit function form $\theta(s, p)=1-e^{-\frac{N_{0}\gamma_{0}\psi(s)}{p}}$ for SDDC model. 
\end{example}

The SDDC in \eqref{eq: SDDC} characterizes a \emph{cyber-physical coupling} between the networked control system $\mathcal{G}$ and the MDP system $\mathcal{M}$. In the presence of such coupling, the first objective of this paper is to find conditions under which system $\mathcal{G}$ achieves \replaced{\emph{stochastic safety} that is formally defined as belows.}{\emph{almost sure asymptotic stability} without external disturbance~($w=0$) and \emph{stochastic stability in probability} with essentially ultimated bounded disturbance. These stability notions are formally defined as follows.}
\begin{definition}[Stochastic Safety]
	\label{def:stochastic-safety}
	Consider the networked control system $\hat{\mathcal{G}}$ in (\ref{eq: G_hat}) and the SDDC model in (\ref{eq: SDDC}), let $\Omega_{s}=\{x \in \mathbb{R}^{n_x+n_c} \vert |x| \leq r \}$ with $r \geq 0$ denote a safe set for $\hat{\mathcal{G}}$ system, and $x_0=x(0)$ denote the initial state of the networked control system, 
	\begin{itemize}
		\item[\textbf{E1}] The $\hat{\mathcal{G}}$ system with $w \equiv 0$ is \emph{asymptotically safe in expectation} with respect to $\Omega_{s}$, if $\forall x(0) \in \Omega_{s}$, there exists a class $\mathcal{KL}$ function $\overline{\beta}(\cdot, \cdot)$ such that
		\begin{align}
		\label{eq: asymptotic-safety-in-expectation}
		\mathbb{E}\big[ |x(t)| \big] \leq \overline{\beta}(|x_0|, t), \quad \forall t \in \mathbb{R}_{\geq 0}
		\end{align}
		and thereby $\lim_{t \rightarrow +\infty}\mathbb{E}\big[ |x(t)| \big]=0$.
		\item[\textbf{E2}] The $\hat{\mathcal{G}}$ system with $|w(t)|_{\mathcal{L}_{\infty}} \leq M_w$ is \emph{asymptotically bounded in expectation} with respect to $\Omega_{s}$, if $\forall x(0) \in \Omega_{s}$, there exists a class $\mathcal{KL}$ function $\overline{\beta}(\cdot, \cdot)$ and a class $\mathcal{K}$ function $\kappa(\cdot)$ such that 
		\begin{align}
		\label{ineq: stochastic-safety-in-expectation}
		\mathbb{E}\big[ |x(t)|\big] \leq \overline{\beta}(|x_0|, t) + \kappa(M_w), \quad \forall t \in \mathbb{R}_{\geq 0}
		\end{align}
		and $\lim_{t \rightarrow +\infty}\mathbb{E}\big[ |x(t)|\big]=\kappa(M_w)$.
		\item[\textbf{P1}] The $\hat{\mathcal{G}}$ system with $w \equiv 0$ is \emph{almost surely asymptotically safe} with respect to $\Omega_{s}$, if $\forall \epsilon, \tau > 0$ and $x_{0} \in \Omega_{s}$, there exists a class $\mathcal{KLL}$ function $\beta_{\epsilon}(\cdot, \cdot, \cdot)$ such that 
		\begin{align}
		\label{ineq: almost-sure-safety}
		{\rm Pr}\big\{\sup_{t \geq \tau} |x(t)| \geq \epsilon+r \big\} \leq \beta_{\epsilon}(|x_0|, \tau, r) 
		\end{align}
		and $\lim_{\tau \rightarrow \infty} {\rm Pr}\big\{\sup_{t \geq \tau} |x(t)| \geq \epsilon+r \big\}=0$.
		\item[\textbf{P2}] The $\hat{\mathcal{G}}$ system with $|w(t)|_{\mathcal{L}_{\infty}} \leq M_w$ is \emph{stochastically safe in probability} with respect to $\Omega_{s}$, if $\forall \epsilon_1 > 0$, there exists a class $\mathcal{KL}$ function $\overline{\beta}_{\epsilon_2}(M_w, r)$ such that 
		\begin{align}
		\label{ineq: stochastic-safety-in-probability}
		\lim_{t \rightarrow \infty}{\rm Pr}\big\{ |x(t)| \geq \epsilon_1+r\big\} \leq \overline{\beta}_{\epsilon_2}(M_w, r).
		\end{align}
	\end{itemize}
	\end{definition}
	

\begin{remark}
The safety notions \textbf{E1} and \textbf{E2} are concerned with system behavior on average~(in the first moment) while the safety notions \textbf{P1} and \textbf{P2} focus on the specification on the sample path of the system. Note that these two types of safety definitions specify both the system's transient and steady behavior. For systems without external disturbance, the safety definition \textbf{E1} requires that the first moment of the norm of the system trajectories must asymptotically converge to the origin if the initial states start within the safety set while the \emph{almost sure asymptotic safety} definition \textbf{P1} is a stronger safety notion than the definition \textbf{E1} in the sense that it requires almost all sample paths starting from the safety set $\Omega_{s}$ stay in the safe region with probability asymptotically going to one. For systems with \replaced{non-vanishing}{non-varnishing} but ultimately bounded disturbance, the definition \textbf{E2} requires that the first moment of the system trajectories is asymptotically bounded with its bound depending on the magnitude of external disturbance. The safety notion \textbf{P2} basically means that the probability of sample paths of the system leaving the safe region is asymptotically bounded and the probability bound is a function of the size of the external disturbance and safety region. These safety notions are closely related to the concepts of stochastic stability defined in \cite{kushner1967,khasminskii2011stochastic}.
\end{remark}

Under the safety conditions for system $\mathcal{G}$, the second objective of this paper is to seek optimal control and communication policies to achieve system efficiency for both system $\mathcal{G}$ and $\mathcal{M}$. A control policy for the MDP system $\mathcal{M}$ is an infinite sequence $\pi^{m}=\{u^{m}_1, u^{m}_2, \ldots\}$ where $u^{m}_k$ is the decision made at time instant $k$. The decision making $u^{m}_k$ is defined as a probability distribution over the action set $A$ given the history information, i.e.,~$u_{k}^{m}={\rm Pr}\{a \vert s_{k}, a_{k-1}, \ldots, s_0\}, \forall a \in A$. Similarly, a power policy for system $\mathcal{G}$ can be defined as $\pi^{p}=\{u^{p}_1, u^{p}_{2}, \ldots, \}$ with $u^{p}_{k}={\rm Pr}\{p \vert s_{k}, a_{k-1}, \ldots, s_0\}$. The policy is \emph{stationary} if $\pi^{m}_{\infty}=\{u_{\infty}^{m}, u_{\infty}^{m}, \ldots \}$~($\pi^{p}_{\infty}=\{u_{\infty}^{p}, u_{\infty}^{p}, \ldots \}$) with $u_{\infty}^{m}={\rm Pr}\{a \vert s\}$~($u_{\infty}^{p}={\rm Pr}\{p \vert s\}$), $\forall a \in A, s \in S$ and $p \in \Omega_p$. \added{This paper will focus on the stationary policy space.}

With the definitions of control $\pi^{m}$ and communication $\pi^{p}$ policies, the system efficiency is defined as a constrained infinite horizon optimization problem as follows,
\begin{equation}
\label{eq: opt-original}
\begin{aligned}
& \underset{\pi^{p}, \pi^{m}}{\text{min}}
& & J_{\alpha}(s_{0}, \pi^{m}, \pi^{p})=(1-\alpha)\sum_{k=0}^{\infty}\alpha^{k}\mathbb{E}\{\lambda c_{p}(p_{k})+c(s_{k}, a_{k})\}\\
& \text{s.t.} 
& & \text{Safety conditions assuring}~\eqref{eq: asymptotic-safety-in-expectation}~\text{or}~\eqref{ineq: stochastic-safety-in-expectation}~\text{or}~\eqref{ineq: almost-sure-safety}~\text{or}~\eqref{ineq: stochastic-safety-in-probability}.
\end{aligned}
\end{equation}
where $c_{p}(\cdot): \Omega_p \rightarrow \mathbb{R}_{\geq 0}$ is the power cost and $c(\cdot, \cdot)$ is the cost defined in the MDP system.  \replaced{$\alpha \in (0, 1)$ is the discounted factor that provides a weight between short term rewards and rewards that might be obtained in a more distance future. $\lambda> 0$ is a parameter used to adjust the weight between communication and control costs.}{$\lambda > 0$ and $\alpha \in (0, 1)$ are the weight and discounted factors respectively.} 

	\section{Stochastic Safety}
\label{sec: safety}
This section presents sufficient conditions to ensure \emph{stochastic safety} defined in Definition \ref{def:stochastic-safety} for the $\mathcal{G}$ system. The following two assumptions are \replaced{needed}{needced} for the main results. 
\begin{assumption}
\label{assumption: iss}
The system $\dot{x}=f(t,x, e, w)$ is input to state stable ~\added{(ISS)} w.r.t. $e$ and $w$, i.e. there exist a class $\mathcal{KL}$ function $\beta(\cdot, \cdot)$, a class $\mathcal{K}$ function $\gamma_{2}(\cdot)$ and a positive real $\overline{\gamma}_{1} \in \mathbb{R}_{\geq 0}$ such that
$
|x(t-t_{0})| \leq \beta(|x(t_{0})|, t-t_{0})+\overline{\gamma}_{1}|e|_{[t_{0},t)}+\gamma_{2}(|w|_{[t_{0}, t)})
$
and $\beta(\cdot, t)$ is a concave function for any fixed $t \in \mathbb{R}_{\geq 0}$. \added{The system is \emph{exponential input to state stable~(Exp-ISS)} w.r.t. $e$ and $w$, if  $\beta(s, t)$ is  a class Exp-$\mathcal{KL}$ function and $\gamma_{2}(s)=\overline{\gamma}_{2}s$ is a linear function with $\overline{\gamma}_{2} > 0$.}
\end{assumption}

\begin{assumption}
\label{assumption: e}
There exists a Lyapunov function $W(\cdot)$ and $\underline{w}, \overline{w}, L_{1}, L_{2}, L_{3} > 0$ for the estimation error dynamics $\dot{e}=g(t, x, e, w)$ in system (\ref{eq: G_hat}) such that 
\begin{align}
\underline{w}|e| &\leq W(e) \leq \overline{w}|e|,  \label{ineq: lower-upper-W}\\
\left\langle \frac{\partial W(e)}{\partial e}, g(t, x, e, w) 
\right\rangle  &\leq L_{1} W(e)+L_{2} |x|+L_{3}|w|.
\label{ineq: grow-W}
\end{align}
\end{assumption}
Assumption \ref{assumption: e} basically requires that the estimation error $e$ is exponentially bounded and the couplings of $x, w$ in the error dynamics are linear.  The following proposition shows that for a given transmission time sequence $\{t_k\}_{k=0}^{\infty}$, the estimation error $e(t_k)$ forms a stochastic jump process whose jump size is $\theta(s, p)$ and depends on the MDP's state $s \in S$ and the transmission power level $p \in \Omega_p$.
\begin{proposition}
\label{proposition: exp-W}
Consider a random dropout process $\{\gamma(k)\}$ associated with the channel's SDDC model in \eqref{eq: SDDC} and let $\{t_{k}\}$ denote the transmission time sequence.
Let $W(e)$ be a Lyapunov function for the error dynamic system in \eqref{eq: G_hat}, then one has
\begin{align}
\mathbb{E}\{W(e(t_{k}^{+})) \Big| s(k)=s, p(k)=p\} = \theta(s, p) W(e(t_{k})).
\end{align} 
where the conditional expectation operator $\mathbb{E}(\cdot \vert \cdot)$ is taken with respect to the random process ${\gamma(k)}$. 
\end{proposition}
\begin{IEEEproof}
The proof is easily completed by combining $W(e(t_{k}^{+}))=(1-\gamma(k))W(e(t_{k}))$ and the SDDC model in \eqref{eq: SDDC}.
\end{IEEEproof}

Under a \emph{state dependent shadow fading channel}, the following theorem present\added{s} a sufficient condition on the Maximum Allowable Transmission Interval (MATI) under which the system $\hat{\mathcal{G}}$ achieves \emph{almost sure asymptotic safety}. In particular, we show that the MATI is a function of the control~($\pi_{\infty}^{m}$) and transmission power~($\pi_{\infty}^{p}$) policies.
\begin{theorem}
\label{thm: almost-sure-safety}
Let $T_{k}=t_{k+1}-t_{k}$ denote the transmission time interval, $P_{m}$ denote the transition matrix defined in \eqref{Pm} and $p \in \Omega_{p}$ denote the transmission power level. Suppose \replaced{the ISS assumption in Assumption \ref{assumption: iss} }{Assumptions \ref{assumption: iss}} and \added{Assumption} \ref{assumption: e} hold, for a given stationary control policy $\pi_{\infty}^{m}$ and a given stationary transmission power policy $\pi_{\infty}^{p}$, the $\hat{\mathcal{G}}$ system with $w=0$ is \replaced{asymptotically safe in expectation~(asymptotically stable in expectation)}{almost surely safe (\textbf{E1} in Definition \ref{def:stochastic-safety})} with respect to the origin, if $T_{k} \in (0, \tau^{*}]$ where
\begin{align}
\tau^{*}=\frac{1}{L_{1}}\ln {\frac{L_{2}\overline{\gamma}_{1}+L_{1}\overline{w}}{L_{2}\overline{\gamma}_{1}+\overline{w}L_{1}\|P_{m}(\pi_{\infty}^{m}, \pi_{\infty}^{p})\text{diag}(\theta(s, p)\|}} > 0
\label{eq: bound_T}
\end{align}
is the MATI. The system parameters $L_{1}$, and $L_{2}$ come from (\ref{ineq: lower-upper-W}) and (\ref{ineq: grow-W}) respectively and
\begin{align}
&\text{diag}(\theta(s, p)) = \begin{bmatrix}
\theta(s_{1}, p_{1}) & \cdots & 0 & \cdots & 0 \\
\vdots & \ddots & \vdots & \ddots & \vdots \\
0 & \cdots & \theta(s_{i}, p_{j}) & \cdots & 0 \\
\vdots & \ddots & \vdots & \ddots & \vdots \\
0 & \cdots & 0 & \cdots & \theta(s_{N}, p_{M})
\end{bmatrix} \nonumber\\
&P_{m}(\pi_{\infty}^{m}, \pi_{\infty}^{p})=\begin{bmatrix}
{\rm Pr}(s_{1}, p_{1} | s_{1}, p_{1}) & \cdots & {\rm Pr}(s_{1}, p_{1} | s_{N}, p_{M}) \\
{\rm Pr}(s_{1}, p_{2} | s_{1}, p_{1}) & \cdots & {\rm Pr}(s_{2}, p_{1} | s_{N}, p_{M}) \\
\vdots & \vdots & \vdots \\
{\rm Pr}(s_{N}, p_{M} | s_{1}, p_{1}) & \cdots & {\rm Pr}(s_{N}, p_{M} | s_{N}, p_{M}) \\
\end{bmatrix}
\label{Pm}
\end{align}
with ${\rm Pr}\{s_{i}, p_{i} \vert s_{j}, p_{j}\}=\sum_{a \in A(s_j)}{\rm Pr}\{s_i \vert a, s_j\}{\rm Pr}\{a \vert s_j\}{\rm Pr}\{p_i \vert s_i\}$.
\end{theorem}
\begin{proof}
The proof is provided in Appendix \ref{appendix: proof}.
\end{proof}
\begin{remark}
The MATI in \eqref{eq: bound_T} generalizes the result in \cite{nevsic2004input}. In particular, one can see that the MATI in \cite{nevsic2004input} is recovered if the shadow fading is absent, i.e., $\theta(s, p)=0, \forall s\in S, p \in \Omega_p$. \deleted{Theorem \ref{thm: almost-sure-safety} also holds for non-stationary policies with the transition matrix $P_m$ being time varying.}
\end{remark}
\begin{theorem}
\label{thm: as-es}
Let the hypothesis in Theorem \ref{thm: almost-sure-safety} and \added{the Exp-ISS assumption in} Assumption \ref{assumption: iss} hold \deleted{and suppose there exists $K_{1}, K_{2}> 0$ such that $\beta(s, t)=K_{1}e^{-K_{2}t} $}, then the system $\mathcal{G}$ is almost surely asymptotically safe~(\textbf{P1} in Definition \ref{def:stochastic-safety}) with respect to the origin. 
\end{theorem}
\begin{proof}
\deleted{The proof is similar to the proof in Theorem \ref{thm: almost-sure-safety} and thus is omitted due to space limit.} \added{The proof is provided in Appendix \ref{appendix: proof}. }
\end{proof}
\begin{theorem}
\label{thm: practical-stability-in-expectation}
Suppose the MATI condition in \eqref{eq: bound_T} holds and consider the system in \eqref{eq: G_hat} with \deleted{essentially bounded external disturbance} $|w|_{\mathcal{L}_{\infty}} \leq M_{w}$,  then the system $\hat{\mathcal{G}}$ is \replaced{asymptotically bounded in expectation~(\textbf{E2} in Definition \ref{def:stochastic-safety})}{stochastically safe in probability} \added{with respect to a bounded safe set $\Omega_{s}=\{x \in \mathbb{R}^{n_x+n_c} \vert |x| \leq r \}$}, i.e., $\forall x(0) \in \Omega_{s}$, there exists a class $\mathcal{KL}$ function $\overline{\beta}(\cdot, \cdot)$ and a class $\mathcal{K}$ function $\kappa(\cdot)$ such that 
\begin{align*}
\mathbb{E}\big[ |x(t)|\big] \leq \overline{\beta}(|x_0|, t) + \kappa(M_w), \quad \forall t \in \mathbb{R}_{\geq 0}
\end{align*}
and $\lim_{t \rightarrow +\infty}\mathbb{E}\big[ |x(t)|\big]=\kappa(M_w)$.
\end{theorem}
\begin{proof}
The proof is provided in Appendix \ref{appendix: proof}.
\end{proof}
\begin{theorem}
Suppose the hypothesis in Theorem \ref{thm: practical-stability-in-expectation} holds, then the system $\hat{\mathcal{G}}$ is stochastically safe in probability~(\textbf{P2} in Definition \ref{def:stochastic-safety}) with respect to a bounded safe set $\Omega_{s}=\{x \in \mathbb{R}^{n_x+n_c} \big\vert |x| \leq r \}$.
\end{theorem}
\begin{IEEEproof}
The result can be straightforwardly obtained by Markov inequality. 
\end{IEEEproof}
	\section{Safety and Efficiency: A Two-player Constrained Cooperative Game}
\label{sec: efficiency}
The \emph{system efficiency} in this paper is defined as an optimization problem where optimal transmission power and control policies are sought to minimize a joint communication and control cost in an infinite horizon. To assure both system efficiency and safety, the control~($\pi^{m}$) and communication~($\pi^{p}$) policies must be carefully coordinated due to their tight couplings as suggested by the safety condition in \eqref{eq: bound_T}. This collaboration between communication and control systems can be naturally formulated as a \emph{two-player constrained cooperative game} where the players' strategy spaces are constrained and coupled. The equilibrium of the game represents the optimal transmission power and control policies to achieve both \emph{system safety} and \emph{efficiency}. 
\begin{problem}[Two-player Constrained Cooperative Game]
\label{problem: CCSG}
Let $c_{p}(\cdot):\Omega_p \rightarrow \mathbb{R}_{\geq 0}$ denote the power cost and $c(\cdot, \cdot): S \times A \rightarrow \mathbb{R}_{\geq 0}$ denote the control cost for the MDP system, the safety and efficiency problem is to find the optimal control $\pi^{m^*}$ and transmission power $\pi^{p^*}$ policies to the following two-player constrained cooperative game,
\begin{equation}
\label{eq: opt}
\begin{aligned}
& \underset{\pi^{p}, \pi^{m}}{\text{min}}
& & J_{\alpha}(s_{0}, \pi^{m}, \pi^{p})\\
& \text{s.t.} 
& & \|P_{m}(\pi^{m}, \pi^{p})\text{diag}(\theta(s, p))\| \leq \xi(T).
\end{aligned}
\end{equation}
where $\alpha \in (0, 1)$ and $T$ is the transmission time interval and $\xi(T) \in (0, 1)$ is a monotonically decreasing function with respect to $T$.
\end{problem} 
\begin{remark}
\label{remark: safety-constraint}
The inequality \eqref{eq: opt} is a safety constraint reformulated by the sufficient condition \eqref{eq: bound_T}. In order to see how this safety constraint is derived from \eqref{eq: bound_T}, let $(\pi^{p}, \pi^{m})$ denote the feasible policies such that $T \leq \tau^{*}(\pi^{p}, \pi^{m})$. Thus
\begin{align*}
T \leq \frac{1}{L_{1}}\ln {\frac{L_{2}\overline{\gamma}_{1}+L_{1}}{L_{2}\overline{\gamma}_{1}+L_{1}\|P_{m}(\pi^{p}, \pi^{m})\text{diag}(\theta(s, p))\|}}.
\end{align*}
By arranging the inequality, one has
\begin{align*}
\|P_{m}(\pi_{p}, \pi_{m})\text{diag}(\theta(s, p))\| \leq \underbrace{\frac{1}{L_{1}}\big [e^{-L_{1} T}(L_{2}\overline{\gamma}_{1}+L_{1})-L_{2}\overline{\gamma}_{1}\big]}_{\xi(T)}.
\end{align*}
Since $T \leq \tau^{*}$, one always has $\xi(T) > 0$. Thus, for any given control $\pi^{m}$ and power $\pi^{p}$ policies that satisfy the above inequality, the sufficient condition in \eqref{eq: bound_T} assures system safety. 
\end{remark}


Under the stationary policy space, we show that the two-player constrained cooperative game Problem \ref{problem: CCSG} can be solved by solving the following constrained nonlinear optimization problem.
\begin{problem}{\bf Constrained Nonlinear Optimization Problem:}
\label{problem: CNOP}
Suppose the state~$S$ and action~$A$ spaces in the MDP system $\mathcal{M}$ are finite sets, and transmission power set $\Omega_{p}$ is finite. Let $u_{\infty}^{p}(p\vert s)={\rm Pr}\{p \vert s\}$ and $ \delta(s, a)$ where $p \in \Omega_p, s \in S$ and $a \in A$, denote the decision variables to the following nonlinear constrained optimization problem. 
\begin{subequations}
\label{opt2}
\begin{align}
& \underset{u^{p}_{\infty}(p \vert s), \delta(s, a)}{\text{min}}
\sum_{(s, a) \in S \times A(s)}  \big ( \lambda \sum_{p \in \Omega_{p}} c_{p}(p)u^{p}_{\infty}(p \vert s) +c(s, a) \big) \delta(s, a) \label{opt2:obj} & &\\
& \text{subject to} & & \nonumber\\
&\sum_{s \in S}\frac{\sum_{a \in A(s)}{\rm Pr}\{s' \vert s, a\}\delta(s, a)}{\sum_{a \in A(s)}\delta(s, a)}\sum_{p \in \Omega_p}u_{\infty}^{p}(p\vert s')\theta(s, p) \leq \xi(T),  \label{opt2: safety-constraint} & &\\
&\sum_{a \in A(s)}\delta(s, a)= D_{0}(s)(1-\alpha)+\alpha \sum_{s' \in S} \sum_{a' \in A(s')}\delta(s', a'){\rm Pr}\{s | s', a' \}, \label{opt2: eq3} & &\\
& \sum_{s} \sum_{a}\delta(s, a)=1, \; \sum_{p \in \Omega_{p}}u^{p}_{\infty}(p \vert s)=1, \delta(s, a)\geq 0, \; u^{p}_{\infty}(p \vert s) \geq 0.  \label{opt2: eq2} & &
\end{align}
\end{subequations}
\end{problem}
The following Lemma shows that Problems \ref{problem: CNOP} and \ref{problem: CCSG} are equivalent in the sense that they have the same optimal solutions and objectives.
\begin{lemma}
\label{lem: problem-equivalence}
Let $\delta^{*}$ and $u^{p^{*}}_{\infty}$ denote the optimal solutions to Problem \ref{problem: CNOP}, then the policies $u^{p^{*}}_{\infty}=\pi^{p^{*}}_{\infty}$ and $\pi^{m^*}_{\infty}(a \vert s)={\rm Pr}\{a|s\}=\frac{\delta^{*}(s, a)}{\sum_{a \in A(s)}\delta^{*}(s, a)}$
are the optimal solutions to Problem \ref{problem: CCSG}.
\end{lemma}
\begin{proof}
The proof can be obtained by examining the equivalence between Problem \ref{problem: CNOP} and Problem \ref{problem: CCSG} in terms of objective function, decision variables and feasible set imposed by the constraints. We have already shown that the objective function in Problem \ref{problem: CCSG} can be rewritten as a function of the new decision variables $\{u_{p}(s, a)\}$ and $\{\delta(s, a)\}$ in Problem \ref{problem: CNOP}. According to the definition of $\delta(s, a)$, one has
$
{\rm Pr}\{a | s\}=\frac{{\rm Pr}\{a , s\}}{{\rm Pr}\{ s\}}
                       =\frac{\delta(s, a)}{\sum_{a \in A(s)} \delta(s, a)}
$. Thus, the decision variable $\delta(s, a)$ uniquely defines the control strategy $\pi^{m}$. The constraints in (\ref{opt2: eq2}) are introduced to enforce the probability law (i.e. non-negativity and total probability being $1$). The constraint in (\ref{opt2: eq3}) is a reformulation of the Markovian dynamics for the MDP in terms of new decision variables $\delta(s, a)$ and $u_{p}(s, a)$ (see \cite{altman1999constrained} for more details). Therefore, one has established the equivalence and the proof is complete. 
\end{proof}
	\begin{remark}
	Problem \ref{problem: CNOP} is a polynomial optimization problem where the objective function and safety constraints in \eqref{opt2: safety-constraint} are polynomial functions. The main challenge to solve this polynomial optimization problem is the fact that the safety constraints are non-convex. The presence of non-convex constraint \eqref{opt2: safety-constraint} in the optimization problem is due to the couplings between communication and control policies in industrial settings with \emph{state-dependent fading wireless channels}.
	\end{remark}
	\subsection{Relaxed Generalized Geometrical Programming}
\label{sec: GGP}
Problem \ref{problem: CNOP} falls into one type of non-convex optimization problem, called Generalized Geometric Program (GGP) \cite{maranas1997global} where the objective function and constraints are the difference of two \emph{posynomials}. A posynomial is a function such that $ G_{i}(x_{1}, x_{2}, \ldots, x_{n})=\sum_{j=1}^{L}a_{ij}x_{1}^{b_{ij1}}x_{2}^{b_{ij2}}\dots x_{n}^{b_{ijn}}$
where $a_{l} > 0, \forall l$ and $b_{ij} \in \mathbb{R}$. 

Let $X=[\delta(s_1, a_1), u_{\infty}^{p}(p_1 \vert s_1), \ldots, \delta(s_N, a_M), u_{\infty}^{p}(p_{\ell} \vert s_N)]^{T}$ denote the decision vector and $\Omega_{X} \subset \mathbb{R}^{NM\ell \times 1}_{+}$ denote the feasible region for $X$. The constrained optimization Problem \ref{problem: CNOP} can be formulated as a GGP as follows,
\begin{equation}
\begin{aligned}
& \underset{X}{\text{minimize}} 
&& G_{0}(X)=G_{0}^{+}(X) \\
& \text{subject to}
&& G_{i}(X)=G_{i}^{+}(X)-G_{i}^{-}(X) \leq 0, \quad i=1, \ldots, N \\
&&& G_{\text{linear}}(X) \leq 0, \quad  X \in \Omega_{X}
\end{aligned}
\end{equation}
where $G_{i}^{+}, G_{i}^{-}, i=1,2, \ldots, N$ are posynomials and $G_{\text{linear}}$ are linear functions. To see how safety constraints in \eqref{opt2: safety-constraint} can be written as the difference of two posynomials, multiplying both sides of \eqref{opt2: safety-constraint} by $\prod_{s \in S}\sum_{a \in A(s)}\delta(s, a)$ leads to 
\begin{align*}
&\underbrace{\sum_{a \in A(s)}{\rm Pr}\{s' \vert a, s\}\delta(s, a)\sum_{p \in \Omega_p}u_{\infty}^{p}(p \vert s')\theta(s, p) \prod_{\tilde{s}\neq s, \tilde{s} \in S}\sum_{a \in A(s)}\delta(\tilde{s}, a)}_{G_{i}^{+}(X)} \\
& - \underbrace{\xi(T)\prod_{s \in S}\sum_{a \in A(s)}\delta(s, a)}_{G_{i}^{-}(X)} \leq 0.
\end{align*}
The above GGP can be further reformulated by introducing an exponential transformation, $X=\exp(Z)$,
\begin{equation}
\label{opt: Z}
\begin{aligned}
& \underset{Z}{\text{minimize}} 
&& \tilde{G}_{0}(Z)=\tilde{G}_{0}^{+}-\tilde{G}_{0}^{-} \\
& \text{subject to}
&& \tilde{G}_{i}(Z)=\tilde{G}_{i}^{+}(Z)-\tilde{G}_{i}^{-}(Z) \leq 0, \quad i=1, \ldots, M \\
&&& G_{\text{linear}}(Z) \leq 0, \quad  Z \in \Omega_{Z}.
\end{aligned}
\end{equation}
where $\Omega_{Z}=\log(\Omega_{X}) \subset \mathbb{R}^{NM\ell \times 1}$, $G_{i}^{-}= \sum_{j \in L_{i}^{-}}a_{ij}\exp{\sum_{l=1}^{n}b_{ijl}z_{l}}$ and $G_{i}^{+}=\sum_{j \in L_{i}^{+}}a_{ij}\exp{\sum_{l=1}^{n}b_{ijl}z_{l}}.$ 

Since $\exp(Z)$ is a convex function in terms of $Z$, $\tilde{G}_{i}^{+}, \tilde{G}_{i}^{-}, i=0,1, \ldots, N$ and $G_{\text{linear}}$ are convex functions as well. However, the function $\tilde{G}_{i}^{+}(Z)-\tilde{G}_{i}^{-}(Z)$ in the safety constraint is generally not convex \cite{maranas1997global}. To address the non-convexity issues, this paper approximates the second terms $\tilde{G}_{i}^{-}$ in the non-convex safety constraints using a linear function. The basic idea is illustrated in Figure \ref{fig: bound-1} using a simple exponential function. In Figure \ref{fig: bound-1}, the linear function shown by the solid line upper approximates the exponential function while the linear function shown by the dashed line approximates the exponential function from below. These two functions can be viewed as upper and lower bounds on the exponential function. The following two subsections are devoted to demonstrate how to construct the upper and lower linear functions for a general multivariate exponential function $\tilde{G}_{i}^{-}(Z)$ for a given domain. 
\subsubsection{Relaxed GGP with Linear Upper Bound}
 For a given bounded domain $\Omega_{Z}=\{Z | Z \in [Z^{L}, Z^{H}]\}$ with $Z^{L}=[z^{L}_{1},\ldots, z^{L}_{n}]$ and $Z^{H}=[z^{H}_{1},\ldots, z^{H}_{n}]$, one can construct a linear function such that,
\begin{align}
&\tilde{G}_{i}^{-}(Z)  \leq A_{i}Z+B_{i} \nonumber\\
A_{i}&=\sum_{j \in L_{i}^{-}}a_{ij}A_{ij}[b_{ij1}, \ldots, b_{ijn}], \quad B_{i}=\sum_{j \in L_{i}^{-}}a_{ij}B_{ij} \label{eq: A-B}\\
A_{ij}&=\frac{\exp(Y_{ij}^{H})-\exp(Y_{ij}^{L})}{Y_{ij}^{H}-Y_{ij}^{L}}, \quad B_{ij}=\frac{Y_{ij}^{H}\exp(Y_{ij}^{L})-Y_{ij}^{L}\exp(Y_{ij}^{H})}{Y_{ij}^{H}-Y_{ij}^{L}} \nonumber\\
Y_{ij}^{L}&=\sum_{l=1}^{n}\min(b_{ijl}z_{l}^{L}, b_{ijl}z_{l}^{H}), \quad Y_{ij}^{H}=\sum_{l=1}^{n}\max(b_{ijl}z_{l}^{L}, b_{ijl}z_{l}^{H}) 
\label{def: Y}
\end{align}
By replacing $\tilde{G}_{i}^{-}(Z)$ with $A_{i}Z+B_{i}, \forall i=0, 1, \ldots, M$  in the transformed GGP (\ref{opt: Z}), one has the convex optimization problem as follows,
\begin{equation}
\label{opt: upper-bound}
\begin{aligned}
& \underset{Z}{\text{minimize}} 
&& \tilde{G}_{0}^{U}(Z)=\tilde{G}_{0}^{+}(Z) \\
& \text{subject to}
&& \tilde{G}_{i}^{U}(Z)=\tilde{G}_{i}^{+}(Z)-(A_{i}Z+B_{i}) \leq 0, \quad i=1, \ldots, N \\
&&& G_{\text{linear}}(Z) \leq 0, \quad  Z \in \Omega_{Z}
\end{aligned}
\end{equation}
Let $\delta_{ij}=Y_{ij}^{H}-Y_{ij}^{L}$ denote the interval width associated with term $j$ in $\tilde{G}_{i}^{-}$ and $\delta_{i}=\max_{j \in L_{i}^{-}}\delta_{ij}$ denote the maximum interval width over all terms in $\tilde{G}_{i}^{-}$. Let $\Delta_{i}(Z)=A_{i}Z+B_{i}-G_{i}^{-}(Z)$ denote the gap between $\tilde{G}_{i}^{-}$ and $A_{i}Z+B_{i}$ and $\Delta_{i}^{*}=\max_{Z \in \Omega_{Z}}\Delta_{i}(Z)$ denote the maximum gap. The following lemma characterizes the explicit relationship between the maximum gap $\Delta_{i}^{*}$ and the size of the region of approximation $\delta_{i}$  \cite{maranas1997global},  
\begin{lemma}
\label{lem: maximum-gap}
Consider the transformed posynomial functions $\tilde{G}_{i}^{-}(Z)$ and its upper approximation $A_{i}Z+B_{i}$ with the region of approximation $\Omega_{Z}$, then, for all $Z \in \Omega_{Z}$, the maximum gap $\Delta_{i}^{*}, \forall i=1, \ldots, N$ defined over $\Omega_{Z}$ is a function of $\delta_{i}$ as follows,
\begin{align*}
\Delta_{i}^{*} &\leq \sum_{j \in L_{i}^{-}}e^{Y_{ij}^{L}}\bigg(1-\Theta(\delta_{ij})+\Theta(\delta_{ij})\log(\Theta(\delta_{ij}))\bigg) \\
                     &\leq |L_{i}^{-}| e^{Y_{i}^{L}}\bigg(1-\Theta(\delta_{i})+\Theta_{i}\log(\Theta(\delta_{i})\bigg)
\end{align*} 
where $e^{Y_{i}^{L}}=\max_{j \in L_{i}^{-}} e^{Y_{ij}^{L}}$ and $\Theta(\delta)=\frac{e^{\delta}-1}{\delta}$. Furthermore, one has
$
\Delta_{i}^{*} \sim \mathcal{O}(\delta_{i}^{2})
$.
\end{lemma}
\begin{proof}
The proof is included in Appendix \ref{appendix: proof}.
\end{proof}
%
\subsubsection{Relaxed GGP with Linear Lower Bound}
Similar to the case of upper bound,a lower bound for the transformed monomial function $\tilde{G}_{ij}^{-}(Z)$ can also be constructed as follows,
\begin{align}
&\tilde{G}_{i}^{-}(Z) \geq A_{i}Z+B_{i}^{L} \nonumber \\
B_{i}^{L}&= \sum_{j \in L_{i}^{-}} a_{ij} A_{ij}\bigg(1-\log(A_{ij})\bigg)
\label{eq: B_{i}^{L}}
\end{align}

By replacing $\tilde{G}_{i}^{-}(Z)$ with $A_{i}Z+B_{i}^{L}, \forall i=0, 1, \ldots, M$  in the transformed GGP (\ref{opt: Z}), one has the following convex optimization with linear lower bounds,
\begin{equation}
\label{opt: lower-bound}
\begin{aligned}
& \underset{Z}{\text{minimize}} 
&& \tilde{G}_{0}^{L}(Z)=\tilde{G}_{0}^{+}(Z)-(A_{0}Z+B_{0}^{L}) \\
& \text{subject to}
&& \tilde{G}_{i}^{L}(Z)=\tilde{G}_{i}^{+}(Z)-(A_{i}Z+B_{i}^{L}) \leq 0, \quad i=1, \ldots, M \\
&&& G_{\text{linear}}(Z) \leq 0, \quad  Z \in \Omega_{Z}
\end{aligned}
\end{equation}
\begin{figure}[!tbp]
\begin{subfigure}{.5\textwidth}
  \centering
  \includegraphics[width=.5\linewidth]{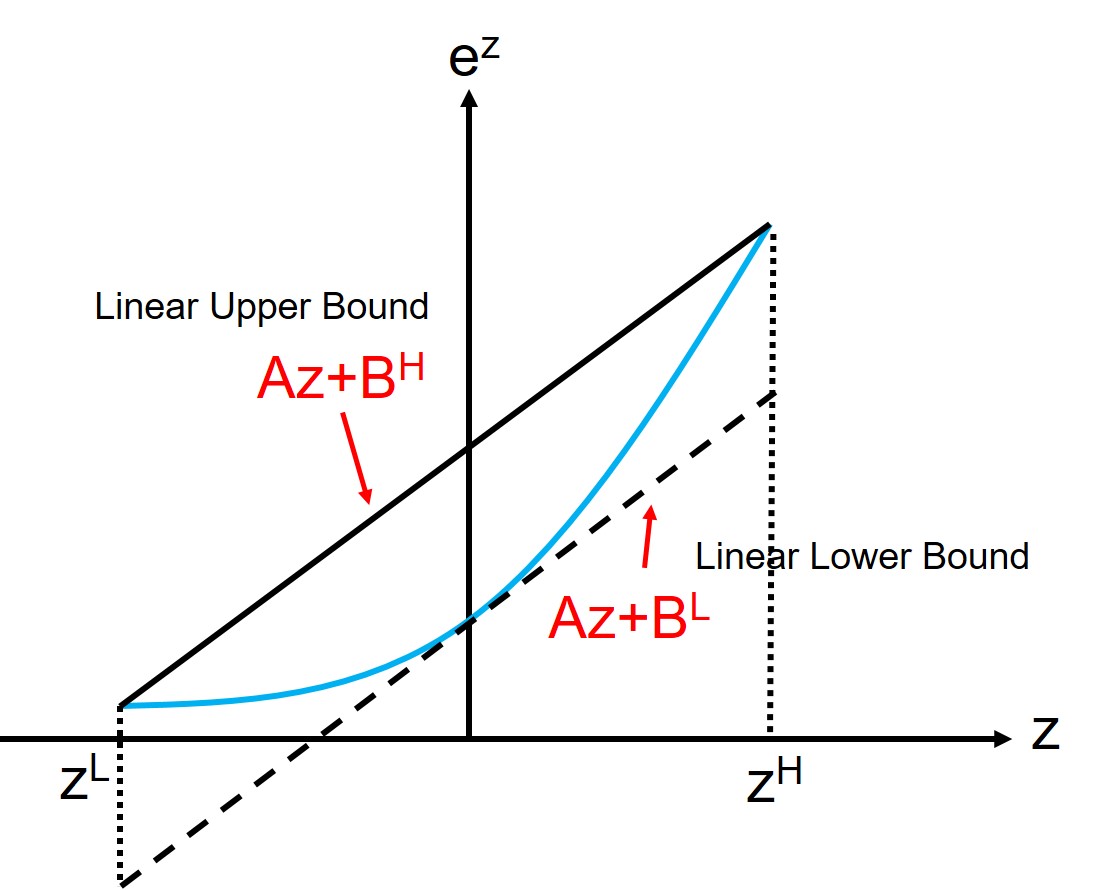}
  \caption{Linear lower and upper approximation}
  \label{fig: bound-1}
\end{subfigure}
\begin{subfigure}{.5\textwidth}
  \centering
  \includegraphics[width=.5\linewidth]{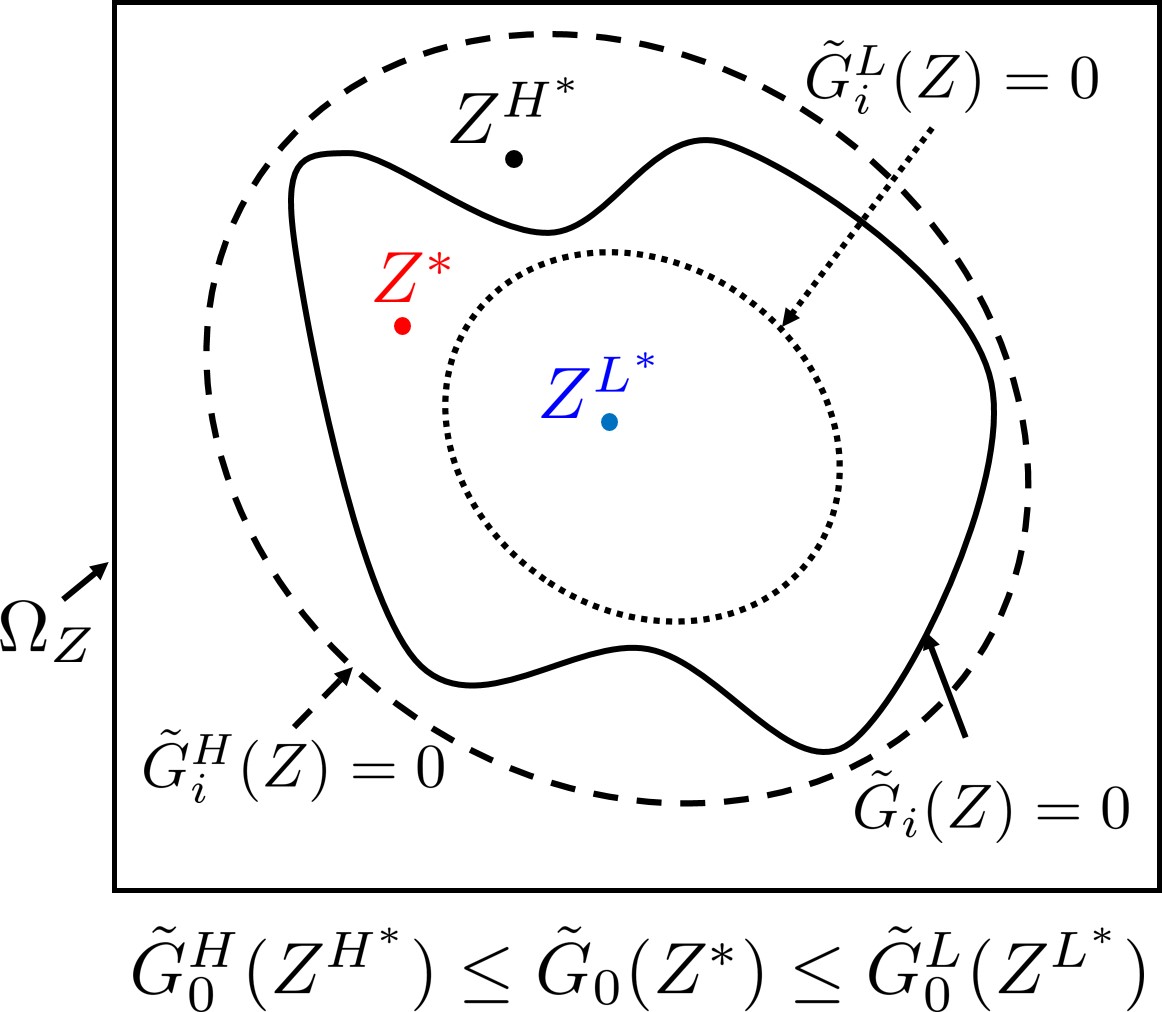}
  \caption{Inner and outer convex regions}
  \label{fig: bound-2}
\end{subfigure}
\caption{Lower and upper bounds by two relaxed convex GGPs}
\label{fig: bound}
\end{figure}

The following lemma shows that the maximum gap of for the lower bound case is the same as the upper bound case. 
\begin{lemma}
\label{lem: lower-bound}
Consider the GGP problem (\ref{opt: Z}) and the relaxed GGP (\ref{opt: lower-bound}) with lower bound linear function $A_{i}Z+B_{i}^{L}, i=0,1,\ldots, M$. Let $\Delta_{i}^{L^{*}}$ denote the maximum gap defined over the domain $\Omega_{Z}$, then 
$\Delta_{i}^{L^{*}} = \Delta^{*}_{i}$
and 
$\Delta_{i}^{L^{*}}=\mathcal{O}(\delta_{i}^{2}) \quad \text{as} \ \delta \rightarrow 0 $.
\end{lemma}
\begin{proof}
The proof is similar to the upper bound case and is omitted here. 
\end{proof}

The following lemma shows that the optimal solutions to the two convex optimizations in (\ref{opt: upper-bound}), (\ref{opt: lower-bound}) are lower and upper bounds to the original non-convex problem in (\ref{opt: Z}).
\begin{lemma}
\label{lem: lower-upper-GGP}
Let $Z^{H^{*}}, Z^{*}$ and $Z^{L^{*}}$ denote the optimal solution to the optimization problems in (\ref{opt: upper-bound}), (\ref{opt: Z}) and (\ref{opt: lower-bound}) respectively, the optimal objective functions then satisfy 
\begin{align}
\label{ineq: obj-ggp}
\tilde{G}_{0}^{H}(Z^{H^{*}}) \leq \tilde{G}_{0}(Z^{*}) \leq \tilde{G}_{0}^{L}(Z^{L^{*}})
\end{align}
and the solution $Z^{L^{*}}$ is a suboptimal solution to the non-convex optimization problem in (\ref{opt: Z}). Let $\overline{\Delta}_{0}:=\tilde{G}_{0}(Z^{L^{*}})-\tilde{G}_{0}(Z^{*})$ denote the gap between the suboptimal and optimal solutions, this gap then has upper upper bound as $ \overline{\Delta}_{0} \leq \tilde{G}_{0}^{L}(Z^{L^{*}})-\tilde{G}_{0}^{H}(Z^{H^{*}})$.
\end{lemma}
\begin{proof}
Let $\mathcal{C}_{v}^{H}$, $\mathcal{C}_{v}$ and $\mathcal{C}_{v}^{L}$ denote the feasible sets that are generated by the constraints in optimization problems (\ref{opt: upper-bound}), (\ref{opt: Z}) and (\ref{opt: lower-bound}) respectively. Since $\mathcal{C}_{v}^{L} \subset \mathcal{C}_{v} \subset \mathcal{C}_{v}^{H}$ and $\tilde{G}_{0}^{H}(Z) \leq \tilde{G}_{0}(Z) \leq \tilde{G}_{0}^{L}(Z)$ hold for any $Z \in \Omega_{Z}$, then one has $\tilde{G}_{0}(Z^{L^{*}}) \leq \tilde{G}_{0}^{L}(Z^{L^{*}})$.
By the definition of $Z^{*}$ and $\mathcal{C}_{v}^{L} \subset \mathcal{C}_{v}$, one further has $
\tilde{G}_{0}(Z^{*}) \leq \tilde{G}_{0}(Z^{L^{*}}) \leq \tilde{G}_{0}^{L}(Z^{L^{*}})$.
The same argument can also be applied to prove $\tilde{G}_{0}^{H}(Z^{H^{*}}) \leq \tilde{G}_{0}(Z^{*}) $.  By Inequality (\ref{ineq: obj-ggp}), the final result holds.
\end{proof}
\subsection{Branch-Bound Algorithm}
This section presents a \emph{Branch-Bound} method under which the lower and upper bounds of the non-convex GGP Problem in \eqref{opt: Z} asymptotically approaches the optimal solutions.
\subsubsection{Branch Procedure}
The branch procedure involves partitioning the hyper-rectangular domain $\Omega_{Z}$ into two small sub-regions under which two convex optimization problems in (\ref{opt: lower-bound}) and (\ref{opt: upper-bound}) are solved. Let $\Omega_{Z}^{i,j}=\{Z \in \mathbb{R}^{n} | Z \in [Z^{L, ij}, Z^{H, ij}]\}$ denote the $j^{th}$ ($j=1,2$) sub-region at the $i^{th}$ stage, where $Z^{L, ij}$ and $Z^{H, ij}$ represent the boundaries of the rectangular constraint $\Omega_{Z}^{i, j}$. For the sub-region $\Omega_{Z}^{i,j}$, let $Z^{H^{*}, ij}, Z^{L^{*}, ij}$ denote the optimal solutions to the problems in (\ref{opt: upper-bound}) and (\ref{opt: lower-bound}). Then, $\tilde{G}_{0}(Z^{H^{*}, ij})$ and $\tilde{G}_{0}(Z^{L^{*}, ij})$ are the corresponding lower bound and upper bound on  $\tilde{G}_{0}(Z^{*})$. Clearly, the upper bound solutions $Z^{L^{*}, ij}$ are always feasible for the original GGP problem while the lower bounds $Z^{H^{*}, ij}$ are not necessarily feasible solutions. In order to obtain tight bounds, the upper and lower bounds are iteratively updated by
\begin{align}
\label{eq: upper-bound}
&\tilde{G}^{UB}_{0}=\begin{cases}
\min \Big\{\tilde{G}_{0}^{UB, i-1}, \tilde{G}_{0}^{LB, i-1}, \{\tilde{G}_{0}(Z^{L^{*}, ij})\}_{j=1,2}\Big\}, \\
\hfill \text{if $\tilde{G}^{LB, i-1}$  is feasible}\\
\min \Big\{\tilde{G}_{0}^{UB, i-1}, \{\tilde{G}_{0}(Z^{L^{*}, ij})\}_{j=1,2}\Big\}, \quad \text{Otherwise}
\end{cases} \\
\label{eq: lower-bounds}
&\tilde{G}^{LB, i}_{0}=\begin{cases}
\big\{\tilde{G}^{LB, i-1}_{0}, \tilde{G}_{0}(Z^{H^{*}, ij})\big\}, \quad \text{if $\tilde{G}_{0}(Z^{H^{*}, ij}) < \tilde{G}^{UB}_{0}, j=1, 2$} \\
\tilde{G}^{LB, i-1}_{0}, \quad \text{Otherwise}
\end{cases}
\end{align}  
The  upper bound $\tilde{G}^{UB}_{0}$ in \eqref{eq: upper-bound} is the minimum feasible solutions up to stage $i$. The $\tilde{G}^{LB, i}_{0}$ is a set of all possible lower bounds that could be used to approach the global optimum. At each stage, the branch procedure selects the region that has the minimum lower bounds, i.e.
\begin{align}
\label{eq: branch-selection}
(l, j)=\arg \min_{0 \leq l \leq i, j=1, 2} \tilde{G}^{LB, i}_{0}
\end{align}
where $(l, j)$ represents the index of the selected region. Thus, the  ``best'' lower bound up to stage $i$ is
\begin{align}
\tilde{G}^{LB}_{0}=\tilde{G}_{0}(Z^{H^{*}, l j})
\label{eq: lower-bound}
\end{align}
The selected region is then partitioned into two smaller regions $\Omega_{Z}^{i+1, j}, j=1,2$ by a bisection of the longest side of the hyper-rectangular.
Two convex optimization problems in (\ref{opt: lower-bound}) and (\ref{opt: upper-bound}) are then constructed based on the new regions $\Omega_{Z}^{i+1, j}, j=1,2$. The lower bound set $\tilde{G}_{0}^{LB, i}$ is further updated by removing current ''best'' lower bound $\tilde{G}_{0}(Z^{H^{*}, l j})$,
\begin{align}
\tilde{G}_{0}^{LB, i} \gets \tilde{G}_{0}^{LB, i} \setminus \tilde{G}_{0}(Z^{H^{*}, l j})
\end{align}
This branch procedure repeats until the gap between the lower and upper bounds is smaller than some specified threshold $\epsilon_{c}$. 
\subsubsection{Bound Procedure}
The bound procedure is to cut those branches that have no feasible solutions or do not contain the global optimum. The criteria to determine which branch can be safely fathomed are based on the monotonicity analysis for the structure of the constraint and objective functions in the relaxed GGP formulation (\ref{opt: lower-bound}). To be specific, consider the following lower bounds for the original $\tilde{G}_{i}^{U}(Z)$, $\forall i=0,1,\ldots, m$ and $\forall Z \in \Omega_{Z}^{\ell, j}$
\begin{align}
\tilde{G}_{i}^{U}(Z) \geq \underline{\tilde{G}}_{i}^{U}  \tilde{G}_{i}^{+}(Y_{i}^{L, \ell j})-\sum_{k \in L_{-i}}a_{ik}A_{ik}^{\ell j}Y_{i}^{H, \ell j}-B_{i}.  
\label{ineq: feasibility}
\end{align}
where $Y_{i}^{L, \ell j}, Y_{i}^{H, \ell j}$ and  $A_{ik}^{\ell j}$ are defined in (\ref{def: Y}) for the region $\Omega_{Z}^{\ell, j}$. $B_{i}$ is defined in (\ref{eq: A-B}). A branch associated with the above bounds can be removed if
\begin{itemize}
\item there exists any $\ell$ such that for any $i \in [1,2, \ldots, m], j \in \{1, 2\}$, the bounds in (\ref{ineq: feasibility}) are positive
\item there exists any $\ell$ such that $\underline{\tilde{G}}_{0}^{U} \geq \tilde{G}_{0}^{UB}$
\end{itemize} 
\begin{remark}
The first condition is used to test whether the convex domain generated by the branch procedure contains any feasible solutions, which is a necessary condition for feasibility test. The second condition is used to eliminate branches that do not contain the global optimum. 
\end{remark}
\subsubsection{Sub-optimality and Distance to Global Optimality}
Obtaining an exact global optimum for a non-convex optimization problem is generally NP-hard \cite{semprog}, which means that ``brute force'' type of searching algorithms are necessary to find global solutions. Hence, it is reasonable to expect suboptimal solutions but with certain performance guarantee. Here, the performance refers to the explicit distance characterization between optimal solutions and suboptimal solutions generated by the branch-bound method. Specifically, we show that the optimality gap can be predicted by measuring the maximum size of the super-rectangular where the sub-optimal solutions locate. This prediction gives rise to an upper bound on the maximum number of stages needed in the Branch-Bound algorithm to achieve the desired optimality gap. 
\begin{theorem}
\label{theorem: GGP}
Consider the non-convex GGP problem in (\ref{opt: Z}), relaxed convex problems in (\ref{opt: upper-bound}) and (\ref{opt: lower-bound}) and the Branch-Bound algorithm, let $Z^{*}$, $Z^{H^{*}}$ and $Z^{L^{*}}$ denote the optimal solutions for the optimization problems (\ref{opt: Z}), (\ref{opt: upper-bound}) and (\ref{opt: lower-bound}) respectively, let $
\delta \coloneqq \max_{1 \leq i \leq m, j \in L_{-i}}(Y_{ij}^{H}-Y_{ij}^{L} )$ denote the maximum size of the super-rectangular region, then the suboptimal solutions $Z^{H^{*}}$ and $Z^{L^{*}}$ asymptotically converge to optimal solution $Z^{*}$ as the maximum size $\delta \rightarrow 0$. Moreover, if the constraint qualification $\exists h_{i} \in \mathbb{R}^{n},  \nabla \tilde{G}_{i}(Z^{*}) h_{i} < 0,  \forall i=1, 2, \ldots, M$ holds at $Z^{*}$, then, one has
\begin{align}
|Z^{H^{*}}-Z^{*}|=\mathcal{O}(\delta) \quad \text{as} \ \delta \rightarrow 0 \label{eq: Z^{H^{*}}-Z^{*}}\\
|Z^{*}-Z^{L^{*}}|=\mathcal{O}(\delta) \quad \text{as} \ \delta \rightarrow 0 
\label{eq: Z^{*}-Z^{L^{*}}}
\end{align} 
Furthermore, let $D_{B}$ denote the depth of a full binary tree generated by the BB algorithm,  then the maximum $D_{B}$ to achieve a desired optimality gap $\delta^{*}$ is 
$
D_{B} \sim \log_{2}(\ceil[\Big]{\frac{\delta^{0}}{\delta^{*}}}^{n}+1)
$ where $\delta^{0}$ is the maximum size of the initial super-rectangular region.
\end{theorem}
\begin{proof}
The proof is provided in Appendix \ref{appendix: proof}.
\end{proof}

	\section{Simulation Results}
\label{sec: example}
This section uses the example of  a two-link planar elbow arm and a forklift truck to demonstrate the effectiveness of our co-design framework in assuring safety and efficiency for factory automation systems.  The \emph{almost sure safety} is demonstrated via Monte Carlo simulations using the sufficient conditions in Theorem \ref{thm: almost-sure-safety}. Under the safety constraint, optimal results regarding the power management for robotic arm system and decision making in forklift trucks are provided to show the system's efficiency as a whole and the necessity of the co-design paradigm. 
\begin{figure}[!tbp]
	\begin{subfigure}{.5\textwidth}
				\centering
	\includegraphics[width=0.4\linewidth]{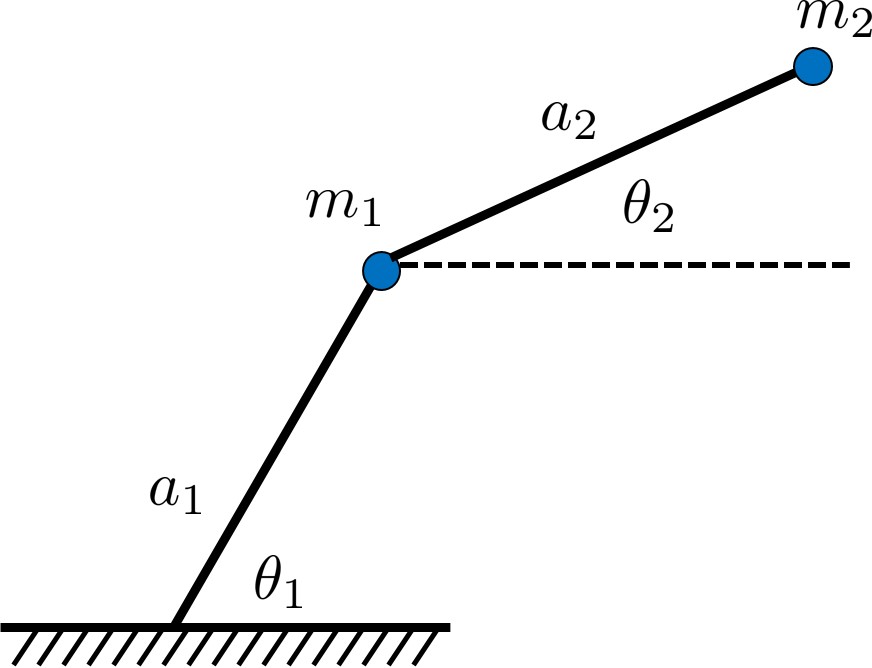}
	\caption{Two Link Planar Elbow Arm}
	\label{fig: arm}
	\end{subfigure}
	\begin{subfigure}{.5\textwidth}
				\centering
		\begin{tikzpicture}[->,>=stealth',shorten >=1pt,auto,node distance=2.2cm,
		  thick,main node/.style={circle,fill=blue!20,draw,font=\sffamily\Large\bfseries}]
		  \node[main node] (1) {$s_{1}$};
		  \node[main node] (2) [right of=1] {$s_{2}$};
		
		  \path[every node/.style={font=\sffamily\small}]
		    (1) edge [bend left] node {$1-u_{1}$} (2)
		        edge [loop above] node {$u_{1}$} (1)
		    (2) edge [bend left] node {$1-u_{2}$} (1)
		        edge [loop above] node {$u_{2}$} (2);
		\end{tikzpicture}
		\caption{The forklift truck: a two-state MDP}
		\label{fig: two-state-MDP}
	\end{subfigure}
	\caption{Simulation example of networked robotic manipulator and forklift truck}
	\label{fig: sim-example}
\end{figure}
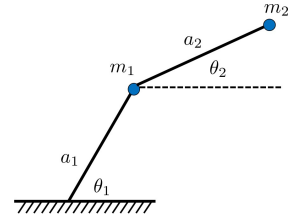
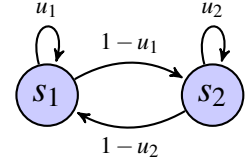

Consider the system dynamics of a nonlinear two-link planar elbow arm as follows \cite{lewis2003robot},
\begin{align*}
&\underbrace{\left [ \begin{array}{c}
(m_{1}+m_{2})ga_{1}\cos{\theta_{1}}+m_{2}ga_{2}\cos(\theta_{1}+\theta_{2}) \\
m_{2}ga_{2}\cos(\theta_{1}+\theta_{2})
\end{array} \right ]}_{G(q)}+ \\
&\underbrace{\left [ \begin{array}{cc}
(m_{1}+m_{2})a_{1}^{2}+a_{2}m_{2}(a_2+2a_{1}\cos{\theta_{2}}) & m_{2}a_2(a_{2}+a_{1}\cos{\theta_{2}}) \\
m_{2}(a_{2}^{2}+a_{1}a_{2}\cos{\theta_{2}}) & m_{2}a_{2}^{2}
\end{array} \right ]}_{M(q)}\begin{bmatrix}    
\ddot{\theta}_{1} \\
\ddot{\theta}_{2}
\end{bmatrix} \\
&=\begin{bmatrix}
\tau_{1} \\
\tau_{2}
\end{bmatrix}-\underbrace{\left [ \begin{array}{c}
	-m_{2}a_{1}a_{2}(2\dot{\theta}_{1}\dot{\theta}_{2}+\dot{\theta}_{2}^{2})\sin{\theta_{2}} \\
	m_{2}a_{1}a_{2}\dot{\theta_{1}}^{2}\sin{\theta_{2}}
	\end{array} \right ]}_{V(q, \dot{q})}
\end{align*}
where $q=[\theta_{1};\theta_{2}]$ are the angles for the upper and lower links of the planar elbow arm as shown in Figure \ref{fig: arm} and $\dot{q}, \ddot{q}$ are the corresponding angular velocities and accelerations. The system inputs $\tau_{i}, i=1, 2$ are the external torque forces that are provided by either motors or hydraulic actuators \cite{lewis2003robot}. These forces $\tau_{i}, i=1,2$ are assumed to be generated by a remote controller, which uses the angular information $q, \dot{q}$ transmitted through a wireless communication channel. With the received angular information, the control objective of the robotic arm is to track a predefined desired trajectory.  

In the simulation, the length $a_{i}$ and mass weight $m_{i}, i=1,2$ for the upper and lower links are set to be $m_{1}=1, a_{1}=2, m_{2}=0.1, a_{2}=10$. The desired angular trajectories are defined as two sinusoidal signal: $q_{d}=[g_{1}\sin(2\pi f_{d}t); g_{2}\sin(2\pi f_{d}t)]$ with desired amplitude $g_{1}=g_{2}=.1$ and frequency $f_d=.5 s^{-1}$. The control input $F=[\tau_{1}; \tau_2]$ is computed by the following feedback linearization method \cite{lewis2003robot},
$
F=M(\hat{q})(\ddot{q}_d-K[\hat{q}-q_d; \hat{\dot{q}}-\dot{q}_d])+V(\hat{q}, \hat{\dot{q}}+G(\hat{q}))
$
where $\hat{q}, \hat{\dot{q}}$ are the estimates of the angular information depending on the real time channel conditions and $K$ is the controller matrix gain 
$
K=[
5, 0, 5, 0; 
0, 5, 0, 5
]
$.

The wireless communication channel used by the robotic arm is subject to shadow fading which is directly related to the physical position of the forklift truck. In the simulation,  the autonomous forklift system is modeled as a two state MDP as shown in Figure \ref{fig: two-state-MDP} where $s_{1}$ is the state representing the good channel region while the state $s_{2}$ characterizes the region causing shadow fading. $u_{i}, i=1,2$ are control strategies characterizing the probabilities of staying in state $s_{i}$ given the current state $s_{i}$, i.e. $u_{i}={\rm Pr}\{``\text{stay}"|s_{i}\}, i=1, 2$. With this state-dependent fading channel, the transmitter in the robotic arm can select high power $p_{H}$ level or low power $p_{L}$ level, to adjust the outage probability as shown in the channel model (\ref{eq: SDDC}). Table \ref{table: outage-probability} shows the outage probabilities $\theta(s, p)$ for different power levels and MDP states
\begin{table}
\centering
\begin{tabular}{c | c | c}
 (Power, State)  &  $s_{1}$ & $ s_{2}$ \\ \hline $ p_{L}$  &   0.4  &  0.9 \\  $p_{H}$  &  0.1  &  0.4  
\end{tabular}
\caption{Outage Probability $\theta(s, p)$ in SDDC (\ref{eq: SDDC})}
\label{table: outage-probability}
\end{table}
\begin{figure}[!tbp]
		\begin{subfigure}{.5\textwidth}
			\centering
			\includegraphics[width=.8\linewidth]{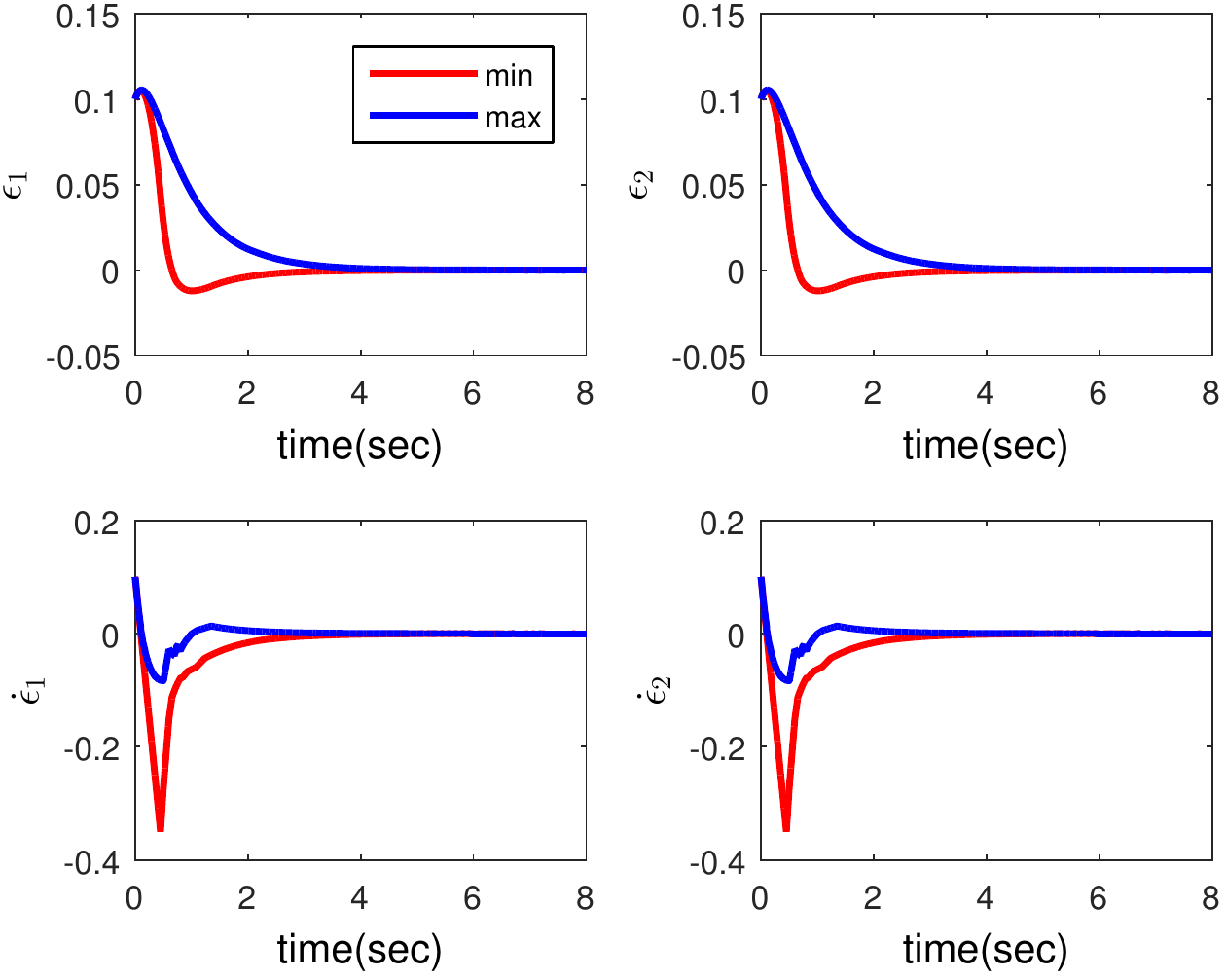}
			\caption{Max. and Min. value of tracking error: $\epsilon=q-q_{d}$}
			\label{fig: max-min-sample-path}
		\end{subfigure}
		\begin{subfigure}{.5\textwidth}
			\centering
			\includegraphics[width=.8\linewidth]{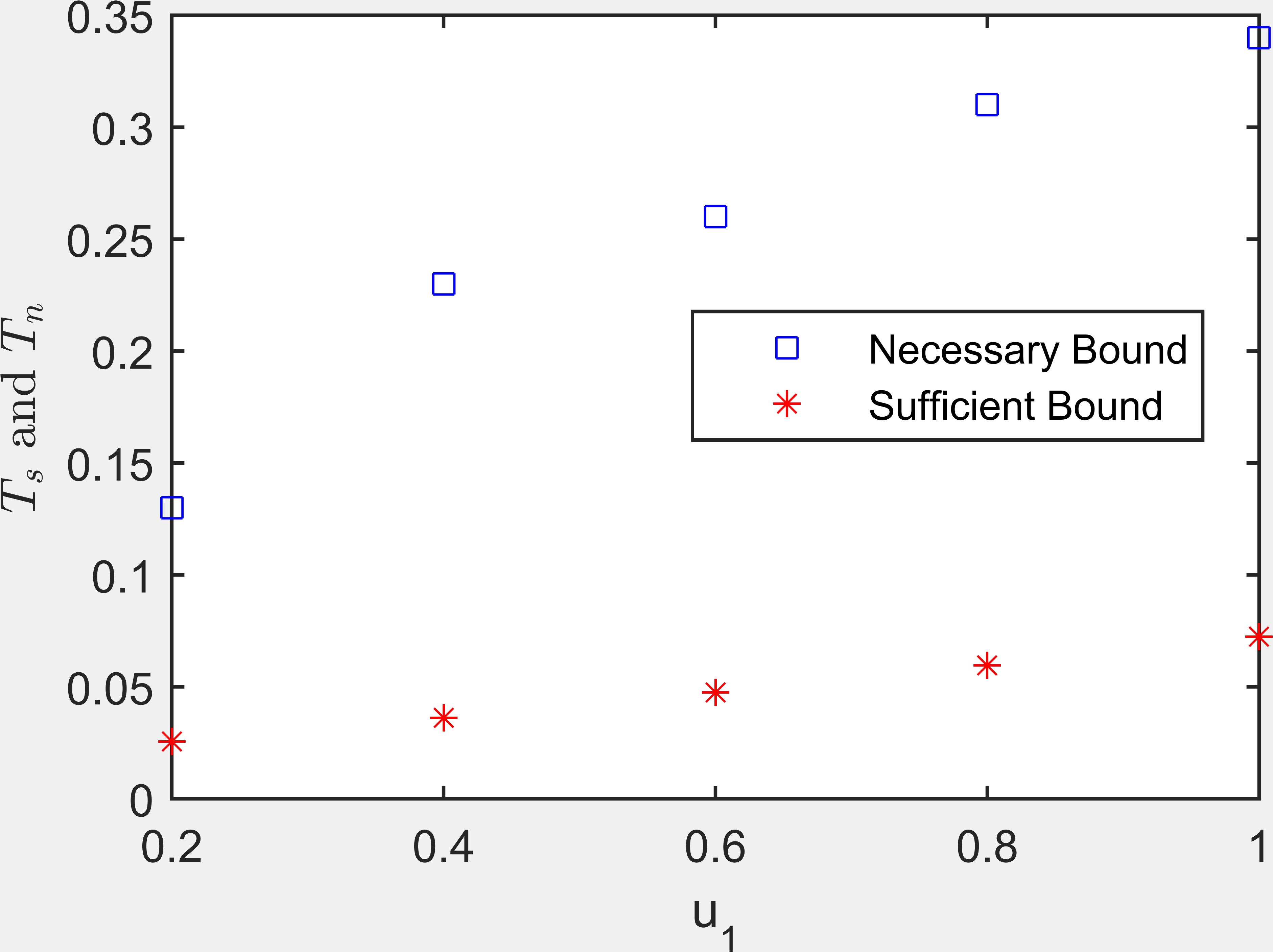}
			\caption{Sufficient~$T_s$ and necessary~$T_n$ bounds on MATI}
			\label{fig: sufficient-necessary}
		\end{subfigure}
			\caption{Almost surely convergence of tracking error on angular states (Left Figure \ref{fig: max-min-sample-path}); Comparison of sufficient and necessary MATI bounds under power and control strategies ${\rm Pr}\{p_{H} | s_{i}, i=1,2\}=1$  and  $u_{1}=0.2, 0.4, 0.6, 0.8, 1$~(Right Figure \ref{fig: sufficient-necessary}).}
\end{figure}
\subsection{Almost Sure Safety}
The first simulation result is to show almost sure safety for the two-link planar elbow arm system under the MATI in \eqref{eq: bound_T} as well as to investigate the tightness of the MATI. A Monte Carlo simulation method is used to generate $1000$ sample paths with each path being evolved over the same time interval from $0$ to $8$ seconds.  

The transmission time interval $T=0.05~s$ is selected to be smaller than the MATI bound $\tau^{*}$ under the control strategy $u_{1}=0.6, u_{2}=0.4$ and the power strategy ${\rm Pr}\{p_{H} | s_{i}, i=1,2\}=1$. Figure \ref{fig: max-min-sample-path} shows the maximum value marked by the blue line, and the minimum value marked by the blue line of the tracking errors ($e_{i}, \dot{e}_{i}, i=1, 2$) over the $1000$ sample paths. One can see from Figure \ref{fig: max-min-sample-path} that the maximum and minimum values of the tracking errors asymptotically converge to zero as time increases. This is precisely the behavior that one would expect if the system is almost surely asymptotically stable. These results, therefore, seem to confirm our sufficient condition in \eqref{eq: bound_T} for almost sure safety.

The tightness of the sufficient conditions is investigated by comparing them against ``necessary'' bounds, which are obtained by an exhaustive search method. This exhaustive search method is to find a lower bound on the MATI such that the system violates the almost sure safety property. The procedure of this searching method starts with the sufficient MATI $\tau^{*}$ bound in \eqref{eq: bound_T}, and then increases the value of $\tau^{*}$ until the maximum and minimum value of the sample path fail to converge to zero. The maximum value of $\tau$ that guarantees almost sure convergence in this search procedure would be the heuristic ''necessary'' bounds.

Figure \ref{fig: sufficient-necessary} shows the comparison of the sufficient MATI bounds~(red stars) obtained by \eqref{eq: bound_T} and necessary MATI bounds~(blue squares) generated by the exhaustive search method under different control strategies $u_{1}=0.2, 0.4, 0.6, 0.8, 1$ and $u_{2}=1-u_{1}$. As shown in the plot, the theoretical sufficient bounds are approximately $5$ times conservative than the heuristic necessary bounds. This performance gap is reasonably close provided that the robotic arm networked system is highly nonlinear. In fact, similar conservativeness (around $6-8$ times) were also reported for deterministic networked systems in \cite{nevsic2004input}. Our results can apply to a more general stochastic networked system but with similar gaps.

Note that $u_{1}$ represents the probability of staying in the good channel region while $u_{2}$ is the probability for the bad channel region where shadow fading occurs. Figure \ref{fig: sufficient-necessary} also shows that the MATI increases when the probability of a good channel increases. This observation matches precisely with the argument that the control strategies from one system do have strong impacts on the network reliability of other systems. This finding motivates our co-design paradigm where the sufficient bound derived in \eqref{eq: bound_T} is a safety constraint that both the robotic arm system and the forklift system must satisfy, to achieve system safety and efficiency.

\subsection{Safety and Efficiency: A Co-design Paradigm}
\label{simulation: efficiency}
This section demonstrates the effectiveness of the co-design paradigm by solving the constrained optimization problem in (\ref{opt2}). In particular, the simulations in this section consist of two parts. The first part is to show that the optimal solutions of the constrained optimization problem (\ref{opt2}) can be achieved asymptotically by solving relaxed convex GGPs using the branch-bound algorithm. The second part of the simulation is to show the necessity of our proposed co-design framework to achieve both system safety and efficiency by comparing it against the separation design framework proposed in \cite{gatsis2014optimal}. 

In the simulation setup, the costs $c(s, a)$ and $c_{p}(p)$ defined in the constrained optimization problem (\ref{opt2}) are shown in Table \ref{table: cost} to simulate a nontrivial scenario where the forklift truck is driven to the bad channel region $s_2$ to achieve its best interest. 
The co-design framework for the example of forklift and robotic arm is formulated as follows,

\begin{table}
\begin{minipage}{.5\textwidth}
\centering
\begin{tabular}{c|c|c}
(Action, State) & $s_{1}$ & $ s_{2}$ \\ \hline Stay &   1.5  &  1. \\  Go  &  .5  &  1 
\end{tabular}
\end{minipage}
\begin{minipage}{.5\textwidth}
\centering
\begin{tabular}{ccc}
Power Level & $p_{H}$ & $p_{L}$ \\ \hline Cost & 2 & .5 
\end{tabular}
\end{minipage}
\caption{Control and Power Cost}
\label{table: cost}
\end{table}

\begin{subequations}
	\label{opt: two-state-MDP}
	\begin{align}
		& \underset{\{x_{i}\}_{i=1}^{4}, \{y_{i}\}_{i=1}^{4}}{\textbf{Minimize}}
		& & \sum_{i=1}^{4}c_{i}x_{i}+2\lambda \sum_{j=1}^{2}(y_{2j-1}c_{p}(p_H) \nonumber \\
		& & & \quad + y_{2j}c_{p}(p_L))(x_{2j-1}+x_{2j}) \\
		& \textbf{subject to}
		& & \begin{cases}
			&(1-\alpha)x_{1}+x_{2}-\alpha x_{4}=(1-\alpha)\delta_{0} \\
			&x_{1}+x_{2}+x_{3}+x_{4}=1,\\
			& y_{2j-1}+y_{2j}=1, j=1, 2\\
			&y_{i} \geq 0, \, x_{i}\geq 0, \quad i=1,2,3,4
		\end{cases} \\
		& & & \begin{cases}
			&\frac{x_{1}y_{1}}{x_{1}+x_2}\theta(s_{1})+\frac{x_{4}y_3}{x_3+x_4}\theta(s_2) \leq c(T) \\
			&\frac{x_{1}y_2}{x_{1}+x_2}\theta(s_{1})+\frac{x_{4}y_4}{x_3+x_4}\theta(s_2) \leq c(T) \\
			&\frac{x_{2}y_1}{x_{1}+x_2}\theta(s_{1})+\frac{x_{3}y_3}{x_3+x_4}\theta(s_2) \leq c(T) \\
			&\frac{x_{2}y_2}{x_{1}+x_2}\theta(s_{1})+\frac{x_{3}y_4}{x_3+x_4}\theta(s_2) \leq c(T)
		\end{cases}
		\label{ineq: safety-constraint}
	\end{align}	
\end{subequations}
where $\{x_{i}\}$ and $\{y_{i}\}$ represent the decision variables related to the control and transmit power policies defined in Problem \ref{problem: CNOP}
$
x_{1}:=\delta(s_{1}, ``\text{Stay}") ,  y_{1}:={\rm Pr}\{p_{H} \vert s_{1}\},
x_{2}:=\delta(s_{1}, ``\text{Go}"), y_{2}:={\rm Pr}\{p_{L} \vert s_{1}\}, 
x_{3}:=\delta(s_{2}, ``\text{Stay}"),  y_{3}:={\rm Pr}\{p_{H} \vert s_{2}\},
x_{4}:=\delta(s_{2}, ``\text{Go}"), y_{4}:={\rm Pr}\{p_{L} \vert s_{2}\}.
$
The inequalities \eqref{ineq: safety-constraint} are the safety constraints and $\theta(s_{i})=\theta(s_{i}, p_{H})+\theta(s_{i}, p_{L})$ is the dropout probability at state $s_{i}$ whose value is shown in Table \ref{table: outage-probability}. As discussed in Section \ref{sec: efficiency}, the parameter $c(T)$ is a function of the transmission time interval $T$ and system parameters in the arm system (See Remark \ref{remark: safety-constraint}). Once $T$ is selected, $c(T)$ is a fixed value. $c_{i}$ is the system cost induced by the state in $x_{i}$, e.g. $c_{2}=c(s_{1}, ``\text{Go}")$ (See Table \ref{table: cost}). The other parameters in the simulation are: $\delta_0={\rm Pr}\{s(0)=s_{1}\}=0.5$, $\alpha=0.8$, and $\lambda=1$. 

By using the GGP formulation and the branch-bound algorithm discussed in Section \ref{sec: GGP}, Figure \ref{fig: lower_bounds}  shows that the lower bounds (blue dashed line) obtained by solving the relaxed convex GGP problem asymptotically approaches the optimal point (red dashed line) as the number of the iteration increases. This result confirms the arguments made in Theorem \ref{theorem: GGP} which state that the global optimal solution is asymptotically achieved by the branch-bound algorithm. 
\begin{figure}
\centering
\includegraphics[width=.8\linewidth]{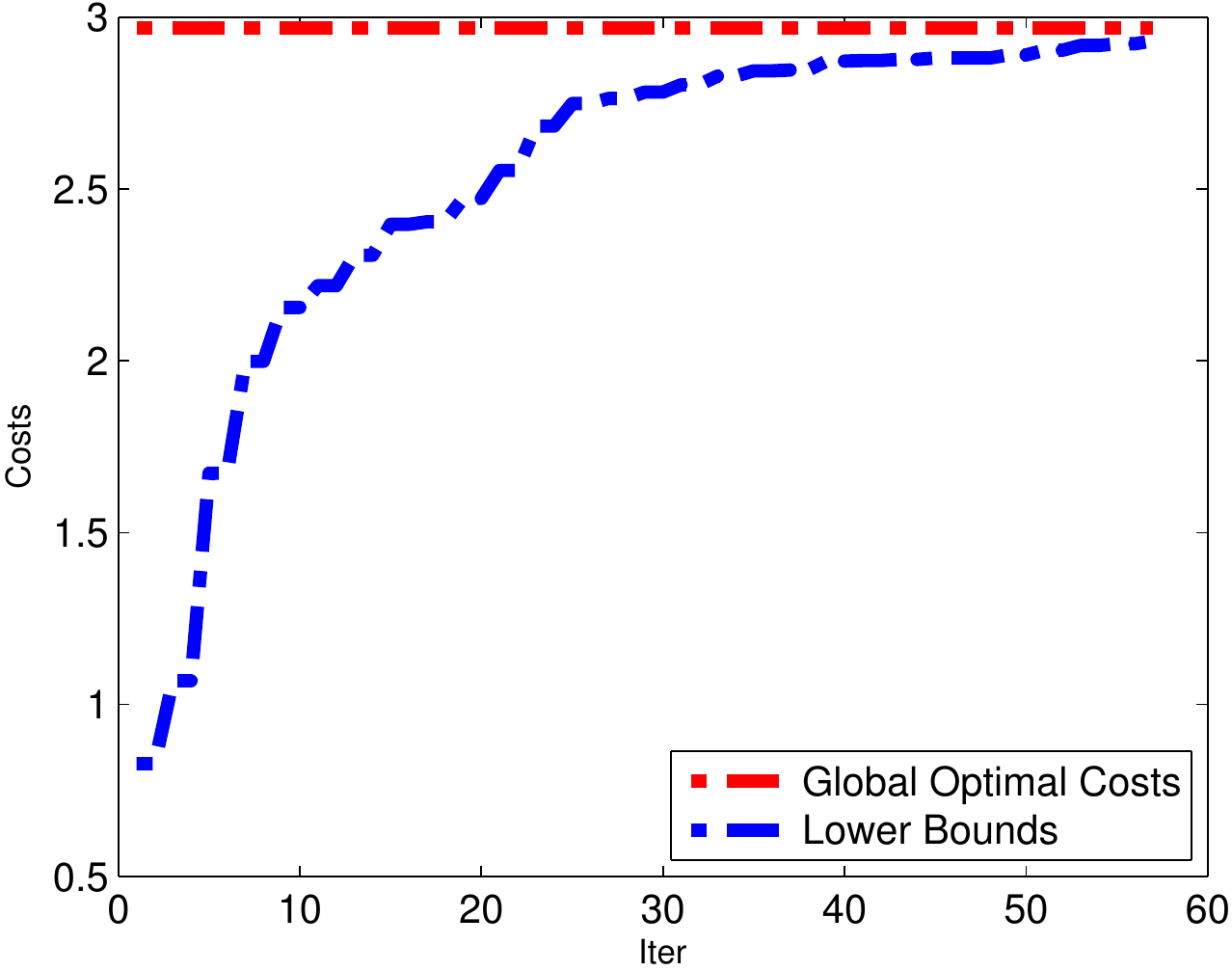}
\caption{Asymptotic Convergence of Lower Bounds to Global Optimum with $T=0.01$ sec}
\label{fig: lower_bounds}
\end{figure}
\begin{figure}
\centering
\includegraphics[width=.8\linewidth]{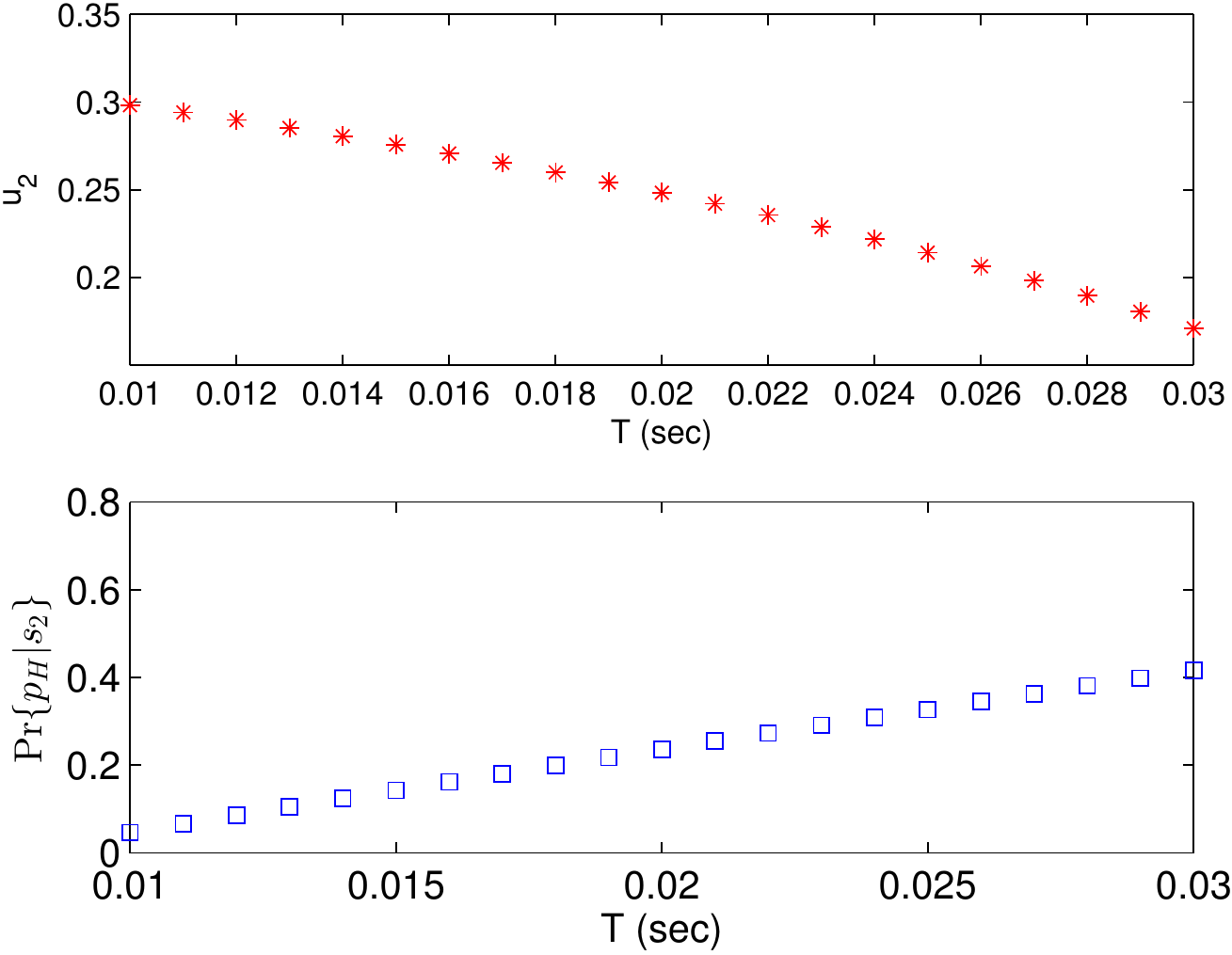}
\caption{Optimal control ($u_{2}:={\rm Pr}\{``\text{Stay}" | s_{2}\}$) and power policies (${\rm Pr}\{p_{H} | s_{2}\}$) at shadowing state $s_{2}$ for different transmission time intervals $T=0.01:0.001:0.03$.}
\label{fig: control_power_T}
\end{figure}

The feasible region enclosed by the safety constraints \eqref{ineq: safety-constraint} shows a tight coupling between the control policies, the transmit power policies, and the transmission time interval $T$. Figure \ref{fig: control_power_T} shows the changes of the optimal control and power strategies as a function of the transmission time interval $T$. In particular, the upper plot of Figure \ref{fig: control_power_T} shows that the optimal control policies $u_2:={\rm Pr}\{``\text{Stay}"| s_{2}\}$~(marked by red stars) for the forklift truck to stay at shadowing state $s_2$ monotonically decreases as the transmission time interval $T$ increases. The bottom plot of Figure \ref{fig: control_power_T} shows that the optimal probabilities~(marked by blue squares) of using high-level transmission power at bad channel region (${\rm Pr}\{p_{H} | s_{2}\}$) increase monotonically as the transmission time interval $T$ increases. These results imply that the overall system safety and efficiency is obtained by active coordinations between communication and control strategies.  

Figure \ref{fig: performance-comparison} shows the performance comparison between the proposed co-design framework and the separation design method under different transmission time intervals $T$~(Figure \ref{fig: comparision-codesign-T}) and different fading levels~(Figure \ref{fig: comparison-codesign-theta}). In this separation design framework, the control and communication polices are designed separately to optimize their own individual interests. In particular, the optimal control policies $u_{i}^{*}, i=1,2$ for the forklift truck are obtained by solving \replaced{ a linear program that is generated by eliminating the decision variables $y_{i}, i=1,2,3,4$ and safety constraint \eqref{ineq: safety-constraint} in the optimization problem \eqref{opt: two-state-MDP}}{the following optimization problem}
and the optimal solutions are $x_{1}^{*}=0, x_{2}^{*}=0.1, x_{3}^{*}=0.9, x_{4}^{*}=0$. Thus, the optimal control policies are $u_{2}^{*}=1, u_{1}^{*}=0$ with the optimal cost $0.55$. On the other hand, the optimal power policies for the robotic arm system are obtained by solving a linear programming problem as below that assumes the worst case impact of the forklift truck system,
	\begin{align}
	& \underset{\{y_{i}\}_{i=1}^{4}}{\textbf{Minimize}}
	& & (0.2y_{1}+1.8y_{3})c_{p}(p_{H})+(0.2y_{2}+1.8y_{4})c_{p}(p_{L}) \nonumber \\
	& \textbf{subject to}
	& & \begin{cases}
	& y_1+y_2=1 \\
	& y_3+y_4=1\\
	& y_1\theta(s_1)+ y_3 \theta(s_2) \leq c(T)\\
	& y_2\theta(s_1)+y_4 \theta(s_2) \leq c(T) \\
	& y_i \geq 0, \quad i=1,2,3,4
	\end{cases}
		\label{opt: separation-power}
	\end{align}	
Note that the safe region generated by the constraints \eqref{opt: separation-power} in the separation design problem is two times smaller than that generated by the co-design framework \eqref{ineq: safety-constraint}. Indeed, the selected transmission time interval $T$ in the co-design framework must satisfy $4c(T) > \theta(s_1)+\theta(s_2)$ to assure that the safe region is nonempty while the condition for the safe region to be nonempty in separation design is $2c(T) > \theta(s_1)+\theta(s_2)$. From the optimization's standpoint, although the co-design framework will for sure lead to better system performance than the separation design method due to its larger safe region, we are interested in investigating how the performance gap evolves as a function of  $T$ and the outage probability $\theta(s_2)$ under the co-design and separation design framework. The sensitivity analysis for these two frameworks is critical to ensuring a robust system design. 
 
Figure \ref{fig: comparision-codesign-T} shows the overall optimal performance (power costs + system costs in MDP) achieved by the co-design~(marked by red dashed line) and separation design~(marked by blue dashed line) methods under the transmission time intervals $T$ ranged from $0.001~sec$ to $0.006~sec$. As expected, the optimal costs generated by the co-design method over the entire time interval are smaller than that under the separation method. Moreover, the performance gap between these two methods increases as the $T$  increases from $0.001$~sec to $0.006$~sec. These results imply that the optimal performance achieved by the co-design method is less sensitive to the changes of $T$ than that achieved by the separation design method. It is worth noting that the constrained optimization problem in \eqref{opt: separation-power} for the separation design method will be infeasible if $T$ is larger than $0.006$ sec.  

\begin{figure}[!tbp]
	\begin{subfigure}{.5\textwidth}
	\centering
	\includegraphics[width=.7\linewidth]{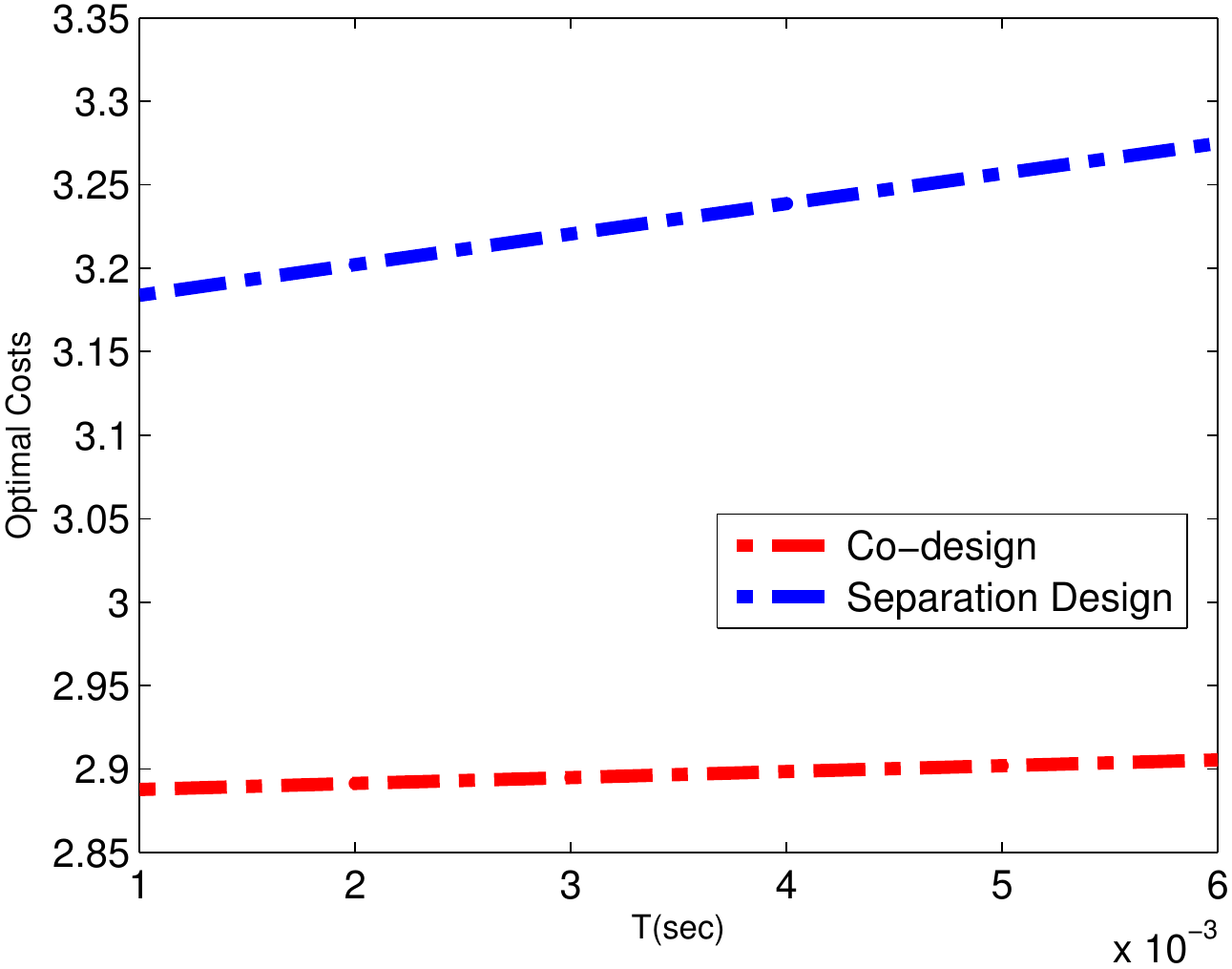}
	\caption{Optimal costs under different $T$}
	\label{fig: comparision-codesign-T}
	\end{subfigure}
		\begin{subfigure}{.5\textwidth}
		\centering
		\includegraphics[width=.7\linewidth]{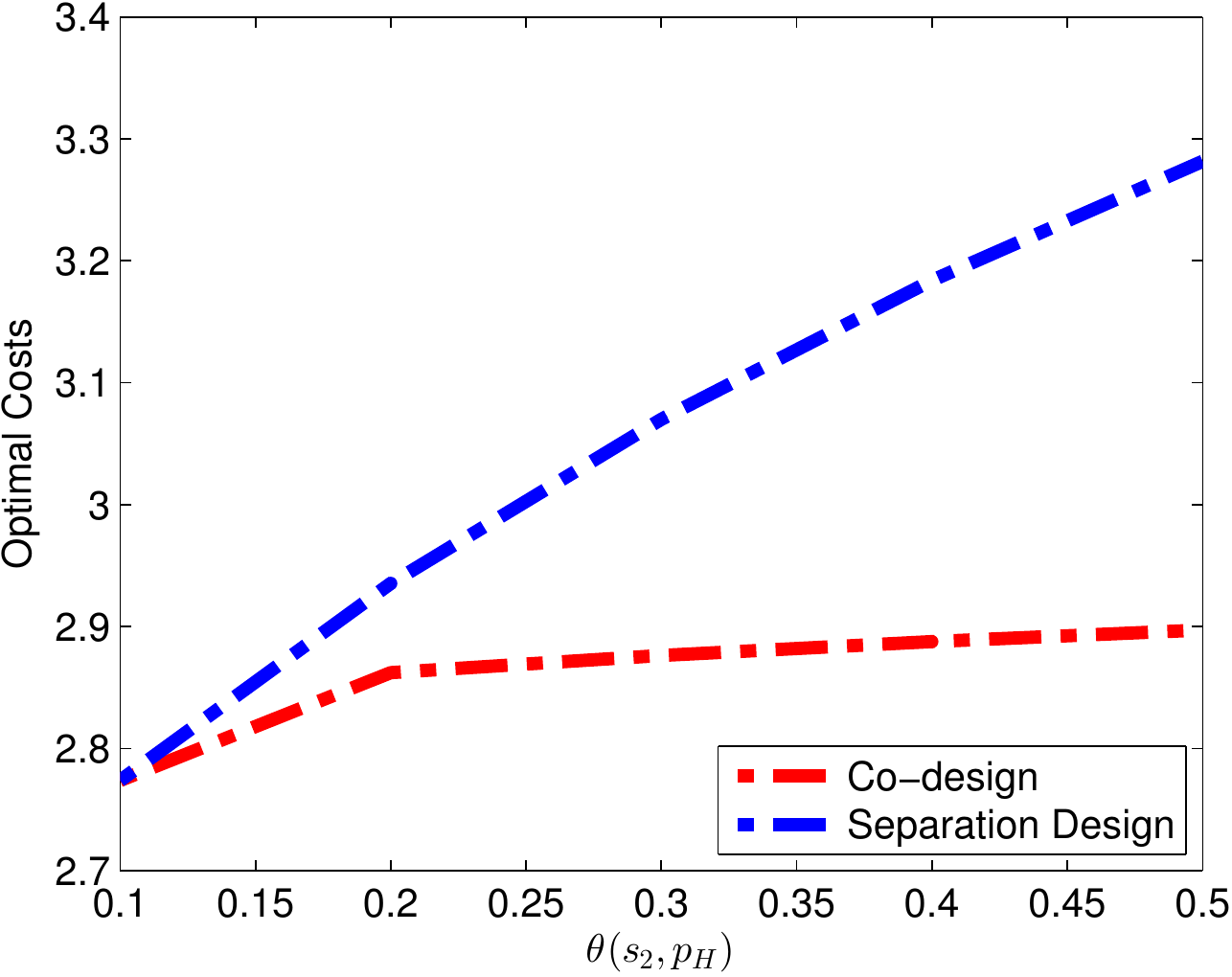}	
		\caption{Optimal costs under  different fading levels $\theta(s_{2}, p_{H})$}
		\label{fig: comparison-codesign-theta}	
		\end{subfigure}
		\caption{The comparison of the optimal performance  achieved by co-design and separation design frameworks under different transmission time intervals $T$ from $0.001$~sec to $0.006$~sec~(Figure \ref{fig: comparision-codesign-T}) and shadow fading levels $\theta(s_{2}, p_{H})=0.1:0.1:0.5$~(Figure \ref{fig: comparison-codesign-theta})} 
		\label{fig: performance-comparison}
\end{figure}

Figure \ref{fig: comparison-codesign-theta} shows the optimal performance comparison under different fading levels. In particular, the shadow fading level is categorized by different outage probability at the shadow state $s_2$. The value of the outage probability $\theta(s_2, p_H)$ at state $s_2$ is selected from $0.1$ to $0.5$ to simulate different levels of shadow fading. As shown in Figure \ref{fig: comparison-codesign-theta}, the optimal costs achieved by the co-design framework (marked by red dashed line) are smaller than the those obtained by the separation design method (marked by blue dashed line) under all fading levels. Furthermore,  the performance gap between these two methods is enlarged as the outage probability $\theta(s_2)$ in the bad channel region $s_2$ increases. In particular, the increase in optimal costs under the co-design framework flattens out even when the fading levels increases dramatically from $0.2$ to $0.5$.  This simulation result suggests that the co-design method is more robust against the shadow fading than the separation method, and is resilient to significant communication degrations. The resilience of the co-design framework is particularly important and useful in factory automation systems where serious shadow fading is often present in wireless links. 
	\section{Conclusion}
	\label{sec: conclusion}
	This paper examines the safety and efficiency of FANs in the presence of a \emph{shadow fading channel} that varies as a function of the physical states. Sufficient conditions on MATI are presented to assure \emph{almost sure asymptotic stability} without external disturbance and \emph{stochastic stability in probability} with non-vanishing external disturbance. These safety conditions are shown to be dependent on the transmission power and the control policies. This observation motivates us to develop a co-design paradigm to ensure system efficiency under the safety constraint. The problem of safety-efficiency co-design is then addressed by solving a two-player constrained cooperative game. Furthermore, we show that the optimal solution to the constraint cooperative game is equivalent to the solution of a non-convex GGP problem. Two relaxed convex GGP were formulated to provide upper and lower bounds on the optimal solution. These bounds asymptotically converge to the global optima by using a branch-bound algorithm. The simulation results of a networked robotic arm and a forklift truck are used to illustrate our findings. 
	
	\added{Our current paper focuses on the safety guarantee for the networked control system~($\mathcal{G}$ system) by co-designing efficient power policies and motion planning policies. It is, however, beyond the scope of this paper, if the objective of the co-design problem also includes ensuring system performance more than safety, e.g.,  optimal tracking control for the networked robotic arms. It is an important and interesting topic that will be pursued in our future work.}
	\appendix
\label{appendix: proof}
\begin{IEEEproof}[Proof of Theorem \ref{thm: almost-sure-safety}]
The techniques used to prove the main results are based on the small gain theorem \cite{isidori1995nonlinear} and Markovian jump system theory \cite{costa2006discrete}. One may view the stochastic hybrid system (\ref{eq: G_hat}) as two interconnected subsystems ($e$ and $x$) modulated with a stochastic jump process ($\{e(t_{k})\}$).  Let $\mathbb{I}_{k} \coloneqq [t_{k}, t_{k+1})$ denote the $k^{th}$ transmission time interval and $T_{k} \coloneqq \tau_{k+1}-\tau_{k}, \forall k \in \mathbb{N}^{+}$ denote the transmission time interval for $\mathbb{I}_{k}$. Consider the error dynamics over $\mathbb{I}_{k}$ and suppose Assumption \ref{assumption: e} holds, one can use comparison principle to bound the error trajectory as
$
W(e(t)) \leq e^{L_{1}(t-t_{k})} W(e(t_{k}^{+})) + \int_{t_{k}}^{t}e^{L_{1}(t-s)}L_{2}|x|ds.
$
Then, one has
\begin{align}
W(e(t_{k+1})) &\leq e^{L_{1}(t_{k+1}-t_{k})} W(e(t_{k}^{+})) + \int_{t_{k}}^{t_{k+1}}e^{L_{1}(t_{k+1}-s)}L_{2}|x|ds \nonumber\\
                      &\leq e^{L_{1}T_{k}}W(e(t_{k}^{+}))+ L_{2}|x|_{[t_{k}, t_{k+1})}\int_{t_{k}}^{t_{k+1}}e^{L_{1}(t_{k+1}-s)} ds \nonumber\\
                      &= e^{L_{1}T_{k}}W(e(t_{k}^{+}))+ \frac{L_{2}}{L_{1}}(e^{L_{1}T_{k}}-1)|x|_{[t_{k}, t_{k+1})}
\label{ineq: W(e(t_{k+1}))}
\end{align}
The second inequality holds because $|x|_{[t_{k}, t_{k+1})} \coloneqq \sup_{t_{k} \leq t < t_{k+1}}|x| \geq |x(\tau)|, \forall \tau \in \mathbb{I}_{k}$. Note that the inequality (\ref{ineq: W(e(t_{k+1}))}) holds for any given initial value $e(t_{k}^{+})$. Moreover, $\{e(t_{k}^{+})\}$ is a stochastic jump process that is governed by stochastic variations on the fading channel. Since the fading channel state in (\ref{eq: SDDC}) depends on the probability measure of MDP state $s$ and power state $p$, let $\mathbbm{1}_{A}$ denote the indicator function that takes value $1$ when sample value falls in set $A$ and takes value $0$ otherwise, then define the operator $W_{k+1}(s, p)\overset{\text{def}}{=}\mathbb{E}\{W(e(t_{k+1})) \mathbbm{1}_{s_{k+1}=s, p_{k+1}=p}\}$ as the expectation of the $W(e(t_{k+1}))$ over the set $\{s_{k+1}=s, p_{k+1}=p\}$. Since $W(e) \geq 0,  \forall e \in \mathbb{R}^{n}$, one can take this expectation operator on both sides of (\ref{ineq: W(e(t_{k+1}))}) without changing the sign,
\begin{align}
&W_{k+1}(s, p) \nonumber\\
&\leq e^{L_{1}T_{k}} \mathbb{E}\{W(e(t_{k}^{+})) \mathbbm{1}_{s, p} \} + \frac{L_{2}}{L_{1}}(e^{L_{1}T_{k}}-1) \mathbb{E}\{|x|_{[t_{k}, t_{k+1})} \mathbbm{1}_{s, p}\} \nonumber\\
&=e^{L_{1}T_{k}} \sum_{s' \in S, p' \in \Omega_{p}} \mathbb{E} \{W(e(t_{k}^{+})) \mathbbm{1}_{s', p'}\} {\rm Pr}\{s, p | s', p'\} \nonumber \\
&\quad +\frac{L_{2}}{L_{1}}(e^{L_{1}T_{k}}-1)|x|_{[t_{k}, t_{k+1})} \mathbb{E}\{\mathbbm{1}_{s, p}\} \label{eq: markovian-property} \\
&=e^{L_{1}T_{k}} \sum_{s' \in S, p' \in \Omega_{p}} W_{k}(s', p') \theta(s', p') {\rm Pr}\{s, p | s', p'\} \nonumber\\
&\quad + \frac{L_{2}}{L_{1}}(e^{L_{1}T_{k}}-1)|x|_{[t_{k}, t_{k+1})} \mathbb{E}\{\mathbbm{1}_{s, p}\} \label{eq: markovian-system}
\end{align}
The first Equality (\ref{eq: markovian-property}) holds due to the Markovian property of the MDP and power processes. The second Equality (\ref{eq: markovian-system}) holds as a result of Proposition \ref{proposition: exp-W} and
$
\mathbb{E} \{W(e(t_{k}^{+})) \mathbbm{1}_{s_k=s', p_k=p'}\}=\mathbb{E}\{W(e(t_{k}^{+})) | s_k=s, p_{k}=p'\} {\rm Pr}\{s_k=s', p_{k}=p'\}.
$
Let $W_{k}\coloneqq [W_{k}(s_{1}, p_{1}), W_{k}(s_{1}, p_{2}), \ldots, W_{k}(s_{|S|, |\Omega_{p}|})]^{T}$, then
\begin{align}
\label{ineq: W-system}
W_{k+1} \leq & e^{L_{1}T_{k}}P_{m}\text{diag}(\theta(s, p)) W_{k} \nonumber \\
&+ \frac{L_{2}}{L_{1}}(e^{L_{1}T_{k}}-1)|x|_{[t_{k}, t_{k+1})} [\mathbb{E}\{\mathbbm{1}_{s, p}\}]
\end{align}
where 
\begin{align*}
&\text{diag}(\theta(s, p)) \coloneqq \begin{bmatrix}
\theta(s_{1}, p_{1}) & \cdots & 0 & \cdots & 0 \\
\vdots & \ddots & \vdots & \ddots & \vdots \\
0 & \cdots & \theta(s_{i}, p_{j}) & \cdots & 0 \\
\vdots & \ddots & \vdots & \ddots & \vdots \\
0 & \cdots & 0 & \cdots & \theta(s_{N}, p_{M})
\end{bmatrix} \\
&P_{m}(\pi_{\infty}^{m}, \pi_{\infty}^{p})\coloneqq \begin{bmatrix}
{\rm Pr}(s_{1}, p_{1} | s_{1}, p_{1}) & \cdots & {\rm Pr}(s_{1}, p_{1} | s_{N}, p_{M}) \\
{\rm Pr}(s_{1}, p_{2} | s_{1}, p_{1}) & \cdots & {\rm Pr}(s_{2}, p_{1} | s_{N}, p_{M}) \\
\vdots & \vdots & \vdots \\
{\rm Pr}(s_{N}, p_{M} | s_{1}, p_{1}) & \cdots & {\rm Pr}(s_{N}, p_{M} | s_{N}, p_{M}) \\
\end{bmatrix}
\end{align*}
and $[\mathbb{E}\{\mathbbm{1}_{s, p}\}] \coloneqq [\mathbb{E}\{\mathbbm{1}_{s_{1}, p_{1}}\} \cdots 	\mathbb{E}\{\mathbbm{1}_{s_{i}, p_{j}}\} \cdots \mathbb{E}\{\mathbbm{1}_{s_{|S|}, p_{|\Omega_{p}|}}\} ]^{T}$.
Since both sides of (\ref{ineq: W-system}) are positive, taking the $\infty$-norm on both sides of (\ref{ineq: W-system}) leads to
\begin{align}
&|W_{k+1}| \nonumber \\
&\leq  e^{L_{1}T_{k}}\underbrace{\|P_{m} \text{diag}(\theta(s, p))\|}_{P_{\infty}}|W_{k}| \nonumber \\
& \quad + \frac{L_{2}}{L_{1}}(e^{L_{1}T_{k}}-1)|\underbrace{[\mathbb{E}\{|x|_{[t_{k}, t_{k+1})} \mathbbm{1}_{s, p}\}]}_{X_{[k, k+1)}}| \nonumber\\
& \leq  \frac{L_{2}}{L_{1}}(e^{L_{1}T^{*}}-1)\Big(|X_{[k, k+1)}|+e^{L_{1}T^{*}}P_{\infty} |X_{[k-1, k)}|+ \cdots \nonumber \\
& \quad + \big(e^{L_{1}T^{*}}P_{\infty}\big)^{k} |X_{[0, 1)}|\Big)+\Big(e^{L_{1}T^{*}}P_{\infty}\Big)^{k+1}|W_{0}| \nonumber \\
&\leq \frac{L_{2}}{L_{1}}(e^{L_{1}T^{*}}-1)\sum_{i=0}^{\infty}\Big(e^{L_{1}T^{*}}P_{\infty}\Big)^{i}|X_{[0, k+1)}| +\Big(e^{L_{1}T^{*}}P_{\infty}\Big)^{k+1}|W_{0}| \nonumber\\
& =\frac{L_{2}}{L_{1}} \frac{e^{L_{1}T^{*}}-1}{1-e^{L_{1}T^{*}}P_{\infty}}|X_{[0, k+1)}|+\Big(e^{L_{1}T^{*}}P_{\infty}\Big)^{k+1}|W_{0}|
\label{ineq: ISS-W}
\end{align}
where $T^{*}=\max_{0 \leq i \leq k}{T_k}$. Clearly, (\ref{ineq: ISS-W}) shows that the $W$ system is input to state stable with respect to $X_{[0, k+1]}$ with linear gain $\frac{L_{2}}{L_{1}}(e^{L_{1}T^{*}}-1) \frac{1}{1-e^{L_{1}T^{*}}P_{\infty}}$ if $e^{L_{1}T^{*}}P_{\infty} < 1$. 

Since $\underline{w}|e| \leq W(e) \leq \overline{w}|e|$, it is straightforward to conclude that the error dynamic system is also input to state stable in expectation as follows, 
\begin{align}
|E_{k+1}| \leq & \frac{L_{2}}{L_{1}\overline{w}}(e^{L_{1}T^{*}}-1) \frac{1}{1-e^{L_{1}T^{*}}P_{\infty}}|X_{[0, k+1)}| \nonumber \\
&+\frac{\underline{w}\Big(e^{L_{1}T^{*}}P_{\infty}\Big)^{k+1}}{\overline{w}}|E_{0}|
\label{ineq: E}
\end{align}
where $E_{k+1}\coloneqq [E_{k+1}(s_{1}, p_{1}), \ldots , E_{k+1}(s_{|S|}, p_{|\Omega_{p}|})]$ with $E_{k+1}(s_{i}, p_{j})=\mathbb{E}\{|e(t_{k+1})| \mathbbm{1}_{s_{i}, p_{j}}\}$.

Similarly, let $X_{t}(s, p) \coloneqq \mathbb{E}\{|x(t)| \mathbbm{1}_{s, p}\}$ denote the expectation of $|x(t)|$ over the set ${s, p}$. By Assumption \ref{assumption: iss} and $\gamma_{1}(s) \leq \overline{\gamma}_{1} s, \forall s > 0$, one has
\begin{align*}
X_{t}(s, p) &\leq \mathbb{E}\{\beta(|x(t_{0})|, t-t_{0})\mathbbm{1}_{s, p}\}+\overline{\gamma}_{1}\mathbb{E}\{|e|_{[t_{0}, t)} \mathbbm{1}_{s, p}\} \\
& \leq \beta(\mathbb{E}\{|x(t_{0})|\mathbbm{1}_{s, p}\}, t-t_0)+\overline{\gamma}_{1}\mathbb{E}\{|e|_{[t_{0}, t)} \mathbbm{1}_{s, p}\} \\
& = \beta(X_{0}(s,p), t-t_0)+\overline{\gamma}_{1}E_{[t_{0}, t)}(s, p)
\end{align*}
then similar to the derivation of (\ref{ineq: E}), one has 
\begin{align}
|X_{t}| \leq \beta(|X_{0}|, t-t_0)+\overline{\gamma}_{1} |E_{[t_{0}, t)}|
\label{ineq: X}
\end{align}
Consider the ISS characterizations of subsystem $X$ in (\ref{ineq: X}) and subsystem $E$ in (\ref{ineq: E}), from the well-established small gain theorem \cite{jiang1994small}, the interconnected system $X$ and $E$ is asymptotically stable if the small gain condition
$\frac{L_{2}}{L_{1}\overline{w}}(e^{L_{1}T^{*}}-1) \frac{1}{1-e^{L_{1}T^{*}}\|P_{m}\text{diag}(\theta(s,p))\|}\overline{\gamma}_{1} < 1 $ holds. 
It is easy to show that the small-gain condition leads to the sufficient condition in (\ref{eq: bound_T}). Since $\mathbb{E}\{|x(t)|\} \leq |X_{t}|, \forall t \geq 0$ and the subsystem $X$ is asymptotically stable, there exists a class $\mathcal{KL}$ function $\overline{\beta}(s, t)$ such that $\mathbb{E}\{|x(t)|\} \leq \overline{\beta}(|x(0)|, t)$. The proof is complete.
\end{IEEEproof}
\begin{IEEEproof}[Proof of Theorem \ref{thm: as-es}]
Under the Exp-ISS assumption in Assumption \ref{assumption: iss}, by following the same argument used in proving Theorem \ref{thm: almost-sure-safety}, one can show that the networked control system $\mathcal{G}$ is exponentially stable in expectation with respect to origin, i.e., there exists a class Exp-$\mathcal{KL}$ function $\beta(s, t)=K_{1}\exp(-K_2 t)s$ such that $\forall x(0) \in \Omega_{s}$, $\mathbb{E}[|x(t)|] \leq K_{1} \exp(-K_{2}t) |x(0)|, \forall t \in \mathbb{R}_{\geq 0}$. Let $\tau' > \tau \geq 0$ denote any time instants such that $\tau \leq t < \tau'$ holds, then for any given $\epsilon > 0$ and the safe set $\Omega_{s}=\{x \in \mathbb{R}^{n_{x}+n_{c}} | |x| \leq r \}$ with $r \geq 0$, consider the following probability bound
\begin{align*}
{\rm Pr}\{& \sup_{\tau \leq t < \tau'} |x(t)| \geq \epsilon+r\} \leq {\rm Pr}\bigg\{\int_{\tau}^{\tau'} |x(t)| dt \geq \epsilon+r \bigg\} \\
& \stackrel{(a)}{\leq} \mathbb{E}\bigg\{ \int_{\tau}^{\tau'} |x(t)| dt \bigg\}/(\epsilon+r) \stackrel{(b)}{\leq} \int_{\tau}^{\tau'} \mathbb{E}\{|x(t)|\} dt /(\epsilon+r) \\
&\leq \int_{\tau}^{\tau'} K_{1} \exp(-K_2t)|x(0)| dt / (\epsilon+r) \\
&\leq \frac{K_1 |x(0)|}{K_2 \epsilon'}[\exp(-K_2 \tau)-\exp(-K_2 \tau')]
\end{align*}
where inequality $(a)$ holds due to the Markov inequality and inequality $(b)$ holds by exchanging the expectation and integration due to the measurability and boundedness of $|x(t)|$ over time interval $[\tau, \tau')$.  Let $\tau' \rightarrow +\infty$, then one has 
\begin{align*}
{\rm Pr}\{\sup_{\tau < t} |x(t)| \geq \epsilon+r \} \leq \frac{K_1 |x(0)|}{K_2 (\epsilon+r)}\exp(-K_2 \tau) \leq \frac{K_1 |x(0)|}{K_2 (\epsilon+r)}.
\end{align*}
Let $\epsilon':= \frac{K_1 |x(0)|}{K_2 (\epsilon+r)}$, and then there exists a function $\delta(\epsilon, \epsilon', r)=\frac{\epsilon' K_2 (\epsilon+r)}{K_1}$ such that 
\begin{align*}
{\rm Pr}\{\sup_{\tau \leq t} |x(t)| \geq \epsilon+r \} \leq \epsilon',  \forall |x(0)| \leq \delta(\epsilon, \epsilon', r).
\end{align*}
Since $\tau \geq 0$ is arbitrarily chosen, by taking $\tau \rightarrow +\infty$, the networked system $\mathcal{G}$ is almost surely asymptotically stable due to 
\begin{align*}
\lim_{\tau \rightarrow \infty}{\rm Pr}\{\sup_{\tau \leq t} |x(t)| \geq \epsilon+r \} \leq \lim_{\tau \rightarrow \infty} \frac{K_1 |x(0)|}{K_2 (\epsilon+r)}\exp(-K_2 \tau)=0.
\end{align*}
The proof is complete.
\end{IEEEproof}
\begin{IEEEproof}[Proof of Theorem \ref{thm: practical-stability-in-expectation}]
Following the argument and notation in the proof of Theorem \ref{thm: almost-sure-safety}, similar to inequalities (\ref{ineq: E}) and (\ref{ineq: X}) one has the interconnected systems with $|w|_{\mathcal{L}_{\infty}} \leq M_w$ defined as follows
\begin{align*}
&|E_{k+1}| \leq \frac{\underline{w}\Big(e^{L_{1}T^{*}}P_{\infty}\Big)^{k+1}}{\overline{w}}|E_{0}|+\frac{L_{2}(e^{L_{1}T^{*}}-1) }{L_{1}\overline{w}(1-e^{L_{1}T^{*}}P_{\infty})}|X_{[0, k+1)}|\\
& \quad +\frac{L_{3}(e^{L_{1}T^{*}}-1) }{L_{1}\overline{w}(1-e^{L_{1}T^{*}}P_{\infty})}|[\mathbb{E}\{|w|_{[t_{k}, t_{k+1})} \mathbbm{1}_{s, p}\}]| \\
&|X_{t}| \leq \beta(|X_{0}|, t-t_0)+\overline{\gamma}_{1} |E_{[t_{0}, t)}|+|[\mathbb{E}\{\gamma_{2} |w|_{[t_{k}, t_{k+1})}\mathbbm{1}_{s, p}\}]|
\end{align*}
Since the small gain condition holds for any transmission time interval $T^{*} \leq \tau^{*}$ where $\tau^{*}$ is defined in (\ref{eq: bound_T}), one can apply the argument in \cite{jiang1994small} to conclude that the composite system state $C_{k} \coloneqq [E_{k}, X_{k}]$ is input to state stable with respect to $w$, i.e., there exists a class $\mathcal{KL}$ function $\overline{\beta}(\cdot, \cdot)$ and a class $\mathcal{K}$ function $\kappa(\cdot)$ such that
$
|C_{k}| \leq \overline{\beta}(|C_{0}|, kT^{*}) + \kappa(M_w)
$.
Given a safe set $\Omega_{s} \coloneqq \{x \in \mathbb{R}^{n}| |x| \leq r \}$, then for any $\epsilon > 0$, the stochastic safety in probability can be characterized as
\begin{align*}
{\rm Pr}\{|x(t)| \geq r+\epsilon\} &\leq \frac{\mathbb{E}(|x(t)|)}{r+\epsilon} \leq \frac{|S||\Omega_{p}||C_{k}|}{r+\epsilon}\\
&\leq |S||\Omega_{p}|\frac{\overline{\beta}(|C_{0}|, kT^{*}) + \kappa(M_w)}{r+\epsilon}
\end{align*}
The first Inequality holds due to the Markov's inequality. The second inequality holds because $\mathbb{E}(|x(t)|)=\sum_{s \in S, p \in \Omega_{p}}X_{t}(s, p) \leq |S||\Omega_{p}||X_{t}|$ and $|X_{t}| \leq |C_{t}|$. One thus has
$
\lim_{t \rightarrow \infty}{\rm Pr}\{|x(t)| \geq r+\epsilon\} \leq |S||\Omega_{p}| \frac{\kappa(M_w)}{r+\epsilon}
$. The proof is complete.
\end{IEEEproof}
\begin{IEEEproof}[Proof of Lemma \ref{lem: maximum-gap}]
\subsubsection{Proof of First Part}: Let $\Delta_{i}=A_{i}Z+B_{i}-\tilde{G}_{i}^{-}(Z)$ denote the gap and $\Delta_{ij}=\overline{A}_{ij}Z+\overline{B}_{ij}-\tilde{G}_{ij}^{-}(Z)$ denote the gap for each term in $\tilde{G}_{i}^{-}(Z)$ where $\overline{A}_{ij}:=a_{ij}A_{ij}[b_{ij1}, \ldots, b_{ijn}], \overline{B}_{ij}:=a_{ij}B_{ij}$ and $\tilde{G}_{ij}^{-}(Z):=a_{ij}\exp{\sum_{l=1}^{n}b_{ijl}z_{l}}, \forall j \in L_{i}^{-}$. Since
\begin{align*}
A_{i}&=\sum_{j \in L_{i}^{-}}\overline{A}_{ij}, \quad B_{i}=\sum_{j \in L_{i}^{-}}\overline{B}_{ij} \\
\tilde{G}_{i}^{-}(Z)&=\sum_{j \in L_{i}^{-}}\tilde{G}_{ij}^{-}(Z), \quad \Delta_{i}=\sum_{j \in L_{i}^{-}} \Delta_{ij}
\end{align*}
and
$
\Delta_{i}^{*}:=\max_{Z \in \Omega_{Z}}\Delta_{i}=\max_{Z \in \Omega_{Z}}(A_{i}Z+B_{i}-\tilde{G}_{i}^{-}(Z)) \leq \sum_{j \in L_{i}^{-}} \max_{Z \in \Omega_{Z}}\Delta_{ij}:= \sum_{j \in L_{i}^{-}}\Delta_{ij}^{*}
$, one can evaluate the maximum gap between the posynomial function and its linear approximation by examining the maximum gap for each term in the posynomial function. Specifically, the maximum point in $\Delta_{ij}$ can be obtained by
$
Z^{*}=\arg \max_{Z \in \Omega_{Z}} \Delta_{ij}(Z) \Leftrightarrow \frac{\partial \Delta_{ij}}{\partial Z}=0
$.
with 
$
\frac{\partial \Delta_{ij}}{\partial Z}=\overline{A}_{ij}-\frac{\partial \tilde{G}_{ij}^{-}(Z)}{\partial Z}, 
\frac{\partial \tilde{G}_{ij}^{-}(Z)}{\partial Z}=a_{ij}[b_{ij1}, b_{ij2}, \ldots, b_{ijn}]e^{\sum_{j \in L_{i}^{-}}b_{ijl}z_l}
$ and $\overline{A}_{ij}:=a_{ij}A_{ij}[b_{ij1}, \ldots, b_{ijn}]$.
Since
$
\sum_{j \in L_{i}^{-}}b_{ijl}z_{l}^{*}=\log{A_{ij}},
\Delta_{ij}^{*}=a_{ij}\bigg(A_{ij} (\log{A_{ij}}-1)+B_{ij}\bigg)
$
with
$
A_{ij}=\frac{\exp(Y_{ij}^{L}+\delta_{ij})-\exp{Y_{ij}^{L}}}{\delta_{ij}}, B_{ij}=\frac{(Y_{ij}^{L}+\delta_{ij})\exp(Y_{ij}^{L})-Y_{ij}^{L}\exp(Y_{ij}^{L}+\delta_{ij})}{\delta_{ij}}
$, one has
$
\Delta_{ij}^{*}=e^{Y_{ij}^{L}}\bigg(1-\Theta_{ij}+\Theta_{ij}\log(\Theta_{ij})\bigg)$ with $\Theta_{ij}=\frac{e^{\delta_{ij}}-1}{\delta_{ij}}$. 
Because $\Delta_{i}^{*} \leq \sum_{j \in L_{i}^{-}} \Delta_{ij}^{*}$, one finally has
\begin{align*}
\Delta_{i}^{*} &\leq \sum_{j \in L_{i}^{-}}e^{Y_{ij}^{L}}\bigg(1-\Theta(\delta_{ij})+\Theta(\delta_{ij})\log(\Theta(\delta_{ij}))\bigg) \\
                     &\leq |L_{i}^{-}| e^{Y_{i}^{L}}\bigg(1-\Theta(\delta_{i})+\Theta_{i}\log(\Theta(\delta_{i})\bigg)
\end{align*} 
where $e^{Y_{i}^{L}}=\max_{j \in L_{i}^{-}} e^{Y_{ij}^{L}}$ and $\Theta(\delta)=\frac{e^{\delta}-1}{\delta}$. The second inequality holds because $\Theta(\delta)$ is a monotonically increasing function with respect to any $\delta \in \mathbb{R}_{\geq 0}$ and $\Theta(\delta_{ij}) \leq \Theta(\delta_{i})$ due to $\delta_{i}=\max_{j \in L_{i}^{-}}\delta_{ij}$. The first part of the proof is complete. 

\subsubsection{Proof of Second Part}: Note that $\Theta_{ij} \rightarrow 1 \Leftrightarrow \delta_{ij} \rightarrow 0$ and the Taylor expansion of function $\log(\Theta_{ij})$ at $\Theta_{ij}=1$ is
$
\log(\Theta_{ij})=(\Theta_{ij}-1)-\frac{1}{2}(\Theta_{ij}-1)^{2}+\frac{1}{3}(\Theta_{ij}-1)^{3}-\cdots
$,
then
$ 1-\Theta_{ij}+\Theta_{ij}\log(\Theta_{ij})=(\Theta_{ij}-1)^{2}\underbrace{(1-\frac{1}{2}\Theta_{ij})}_{>0 \ \text{around} \ \Theta_{ij}=1}
+\Theta_{ij}(\Theta_{ij}-1)^{3}\underbrace{(\frac{7}{12}-\frac{1}{4}\Theta_{ij})}_{ > 0 \ \text{around} \ \Theta_{ij}=1}+ \cdots$.
Taking the Taylor expansion for function $e^{\delta_{ij}}$ at point $0$, one further has
$
\Theta_{ij}(\delta_{ij})-1=\frac{1}{2!}\delta_{ij} +\frac{1}{3!}\delta_{ij}^{2}+\cdots
$.
Thus,  
$
\Delta_{ij}^{*} \sim \mathcal{O}(\delta_{ij}^{2})
$.
Since $\Delta_{i}^{*} \leq \sum_{j \in L_{i}^{-}}\Delta_{ij}^{*}$, then
$
\Delta_{i}^{*} \sim \mathcal{O}(\delta_{i}^{2})
$.
The second part of the proof is complete. 
\end{IEEEproof}
\begin{IEEEproof}[Proof of Theorem \ref{theorem: GGP}]
The proof is based on the perturbation analysis for the non-convex optimization problem \cite{bonnans2013perturbation}. Let the non-convex GGP problem in (\ref{opt: Z}) denote the unperturbed nominal optimization and the relaxed convex problems (\ref{opt: upper-bound}) and (\ref{opt: lower-bound}) denote the perturbed optimization defined as follows,
\begin{equation}
\label{opt: perturbed-optimization}
\begin{aligned}
& \underset{Z}{\text{minimize}} 
&& \tilde{G}_{0}(Z, u_{0})\\
& \text{subject to}
&& \tilde{G}_{i}(Z, u_{i}) \leq 0, \quad i=1, \ldots, M \\
&&& Z \in \Omega_{Z}
\end{aligned}
\end{equation}
where $u_{i}=\tilde{G}_{i}^{s}(Z)-\tilde{G}_{i}(Z), s=H, L$ represents the perturbation term and $|u_{i}| \leq \Delta_{i}^{*}(\delta_{i})$ with $\Delta_{i}^{*}$ defined in Lemma \ref{lem: maximum-gap}. Let $\nu(u)$ denote the optimal value of the perturbed optimization problem in \eqref{opt: perturbed-optimization} which is a function of $u$. By Lemma \ref{lem: maximum-gap} and \ref{lem: lower-bound}, let $\Phi(\delta)=\{u \in \mathbb{R}^{M+1} \big\vert| |u| \leq \Delta(\delta)\}$ denote a compact set with $\delta =\max_{0 \leq i \leq M}\delta_{i}$, the objective is to show how optimal solutions $Z^{*}(u)$ of the perturbed optimization problem in \eqref{opt: perturbed-optimization} and optimal value $\nu(u)$ for any $u \in \Phi(\delta)$ converge to the optimal solutions of the unperturbed problem when $\delta \rightarrow 0$. First, it is easy to show that $Z^{*}(0) = \lim_{\delta \rightarrow 0} Z^{*}(u)$ and $\nu(0)=\lim_{\delta \rightarrow 0}\nu(u)$ since $\delta \rightarrow 0 \Longrightarrow u \rightarrow 0$ and $\tilde{G}_{i}(Z, u_{i}), \forall i=0,1, \ldots, M$ is smooth with respect to both $Z$ and $u$. Furthermore, one knows that $|u| = \mathcal{O}(\delta^{2})$ as $\delta \rightarrow 0$ by Lemma \ref{lem: maximum-gap}. Thus, one can define the perturbation path, along a direction $d \in \mathbb{R}^{M+1}$, in the parameter space $\Phi(\delta)$ as
$
u(\delta)=u_{0}+\delta^{2}d+\mathcal{O}(\delta^{3})
$.
The vector $d$ characterizes the perturbation directions for the constraint functions and objective function in optimization problem in \eqref{opt: perturbed-optimization}. For example, $d=[0, 0, \ldots, \underbrace{+1}_{i+1}, \ldots, M]^{T}$ represents the positive perturbation occurring at $i^{th}$ constraint function. Note that the branch procedure defined in the branch-bound algorithm defines the perturbation direction $d$. Let $Z^{*}$ denote the optimal solution for the unperturbed problem and $\tilde{G}_{0}(Z^{*}, u_{0})$ denote the optimal value.  Let $h_{i}$ denote any feasible direction such that the directional regularity $DG_{i}(Z^{*}, u_{0})(h_{i}, d_{i}) < 0, i=1,2,\ldots, M$ holds (see Section 4.2 in  \cite{bonnans2013perturbation}). Since $u_{0}=0$ represents the unperturbed problem,  for any perturbed problem along the path $u(t)=\delta^{2}d+\mathcal{O}(\delta^{3})$, one has $|Z^{*}(u)-Z^{*}|=\mathcal{O}(\delta)$ by Theorem 4.53 in \cite{bonnans2013perturbation}. It is easy to verify that the lower and upper bounds $Z^{H^{*}}$ and $Z^{L^{*}}$ generated by the branch-bound algorithm correspond to the cases $Z^{*}(u)$ when perturbation directions $d$ are selected oppositely. 

The branch-bound algorithm is a bisection method whose data structure forms a binary tree. By (\ref{eq: Z^{H^{*}}-Z^{*}}) or (\ref{eq: Z^{*}-Z^{L^{*}}}), one can define desired optimality gap as $\delta^{*}$, then for a given initial gap $\delta_{0} > \delta^{*}$, the maximum number of bisections $Nb$ that are needed to achieve $\delta^{*}$ for a $n$-dimensional variable $Z$ is $Nb=\ceil[\Big]{\frac{\delta^{0}}{\delta^{*}}}^{n}$. Note that the relationship between binary tree depth and the number of bisections is $Nb=2^{DB}-1$. One has $DB=\log_{2}(\ceil[\Big]{\frac{\delta^{0}}{\delta^{*}}}^{n}+1)$. Since optimality gap satisfies (\ref{eq: Z^{H^{*}}-Z^{*}}) and (\ref{eq: Z^{*}-Z^{L^{*}}}), one has the final conclusion and the proof is complete. 
\end{IEEEproof}

	
	\bibliographystyle{IEEEtran}
	\bibliography{tase2017}

\begin{thebibliography}{10}
\providecommand{\url}[1]{#1}
\csname url@samestyle\endcsname
\providecommand{\newblock}{\relax}
\providecommand{\bibinfo}[2]{#2}
\providecommand{\BIBentrySTDinterwordspacing}{\spaceskip=0pt\relax}
\providecommand{\BIBentryALTinterwordstretchfactor}{4}
\providecommand{\BIBentryALTinterwordspacing}{\spaceskip=\fontdimen2\font plus
\BIBentryALTinterwordstretchfactor\fontdimen3\font minus
  \fontdimen4\font\relax}
\providecommand{\BIBforeignlanguage}[2]{{%
\expandafter\ifx\csname l@#1\endcsname\relax
\typeout{** WARNING: IEEEtran.bst: No hyphenation pattern has been}%
\typeout{** loaded for the language `#1'. Using the pattern for}%
\typeout{** the default language instead.}%
\else
\language=\csname l@#1\endcsname
\fi
#2}}
\providecommand{\BIBdecl}{\relax}
\BIBdecl

\bibitem{zhuang2007wireless}
L.~Zhuang, K.~M. Goh, and J.-B. Zhang, ``The wireless sensor networks for
  factory automation: issues and challenges,'' in \emph{Emerging Technologies
  and Factory Automation, 2007. ETFA. IEEE Conference on}.\hskip 1em plus 0.5em
  minus 0.4em\relax IEEE, 2007, pp. 141--148.

\bibitem{rajhans2014supporting}
A.~Rajhans, A.~Bhave, I.~Ruchkin, B.~H. Krogh, D.~Garlan, A.~Platzer, and
  B.~Schmerl, ``Supporting heterogeneity in cyber-physical systems
  architectures,'' \emph{Automatic Control, IEEE Transactions on}, vol.~59,
  no.~12, pp. 3178--3193, 2014.

\bibitem{de2006use}
F.~De~Pellegrini, D.~Miorandi, S.~Vitturi, and A.~Zanella, ``On the use of
  wireless networks at low level of factory automation systems,''
  \emph{Industrial Informatics, IEEE Transactions on}, vol.~2, no.~2, pp.
  129--143, 2006.

\bibitem{groover2007automation}
M.~P. Groover, \emph{Automation, production systems, and computer-integrated
  manufacturing}.\hskip 1em plus 0.5em minus 0.4em\relax Prentice Hall Press,
  2007.

\bibitem{islam2012wireless}
K.~Islam, W.~Shen, and X.~Wang, ``Wireless sensor network reliability and
  security in factory automation: A survey,'' \emph{Systems, Man, and
  Cybernetics, Part C: Applications and Reviews, IEEE Transactions on},
  vol.~42, no.~6, pp. 1243--1256, 2012.

\bibitem{tse2005fundamentals}
D.~Tse and P.~Viswanath, \emph{Fundamentals of wireless communication}.\hskip
  1em plus 0.5em minus 0.4em\relax Cambridge university press, 2005.

\bibitem{agrawal2014long}
P.~Agrawal, A.~Ahl{\'e}n, T.~Olofsson, and M.~Gidlund, ``Long term channel
  characterization for energy efficient transmission in industrial
  environments,'' \emph{IEEE Transactions on Communications}, vol.~62, no.~8,
  pp. 3004--3014, 2014.

\bibitem{quevedo2013state}
D.~E. Quevedo, A.~Ahlen, and K.~H. Johansson, ``State estimation over sensor
  networks with correlated wireless fading channels,'' \emph{Automatic Control,
  IEEE Transactions on}, vol.~58, no.~3, pp. 581--593, 2013.

\bibitem{kashiwagi2010time}
I.~Kashiwagi, T.~Taga, and T.~Imai, ``Time-varying path-shadowing model for
  indoor populated environments,'' \emph{IEEE Transactions on Vehicular
  Technology}, vol.~59, no.~1, pp. 16--28, 2010.

\bibitem{gatsis2014optimal}
K.~Gatsis, A.~Ribeiro, and G.~J. Pappas, ``Optimal power management in wireless
  control systems,'' \emph{Automatic Control, IEEE Transactions on}, vol.~59,
  no.~6, pp. 1495--1510, 2014.

\bibitem{tatikonda2004control}
S.~Tatikonda and S.~Mitter, ``Control over noisy channels,'' \emph{Automatic
  Control, IEEE Transactions on}, vol.~49, no.~7, pp. 1196--1201, 2004.

\bibitem{elia2005remote}
N.~Elia, ``Remote stabilization over fading channels,'' \emph{Systems \&
  Control Letters}, vol.~54, no.~3, pp. 237--249, 2005.

\bibitem{zhang1999finite}
Q.~Zhang, S.~Kassam \emph{et~al.}, ``Finite-state markov model for {R}ayleigh
  fading channels,'' \emph{Communications, IEEE Transactions on}, vol.~47,
  no.~11, pp. 1688--1692, 1999.

\bibitem{wang1995finite}
H.~S. Wang and N.~Moayeri, ``Finite-state markov channel-a useful model for
  radio communication channels,'' \emph{Vehicular Technology, IEEE Transactions
  on}, vol.~44, no.~1, pp. 163--171, 1995.

\bibitem{agrawal2009correlated}
P.~Agrawal and N.~Patwari, ``Correlated link shadow fading in multi-hop
  wireless networks,'' \emph{IEEE Transactions on Wireless Communications},
  vol.~8, no.~8, 2009.

\bibitem{leong2016network}
A.~S. Leong, D.~E. Quevedo, A.~Ahl{\'e}n, and K.~H. Johansson, ``On network
  topology reconfiguration for remote state estimation,'' \emph{IEEE
  Transactions on Automatic Control}, vol.~61, no.~12, pp. 3842--3856, 2016.

\bibitem{hu2013using}
B.~Hu and M.~D. Lemmon, ``Using channel state feedback to achieve resilience to
  deep fades in wireless networked control systems,'' in \emph{Proceedings of
  the 2nd ACM international conference on High confidence networked
  systems}.\hskip 1em plus 0.5em minus 0.4em\relax ACM, 2013, pp. 41--48.

\bibitem{hu2015distributed}
------, ``Distributed switching control to achieve almost sure safety for
  leader-follower vehicular networked systems,'' \emph{IEEE Transactions on
  Automatic Control}, vol.~60, no.~12, pp. 3195--3209, 2015.

\bibitem{caire1999optimum}
G.~Caire, G.~Taricco, and E.~Biglieri, ``Optimum power control over fading
  channels,'' \emph{Information Theory, IEEE Transactions on}, vol.~45, no.~5,
  pp. 1468--1489, 1999.

\bibitem{goldsmith1997variable}
A.~J. Goldsmith and S.-G. Chua, ``Variable-rate variable-power {MQAM} for
  fading channels,'' \emph{Communications, IEEE Transactions on}, vol.~45,
  no.~10, pp. 1218--1230, 1997.

\bibitem{quevedo2014power}
D.~E. Quevedo, J.~Ostergaard, and A.~Ahlen, ``Power control and coding
  formulation for state estimation with wireless sensors,'' \emph{Control
  Systems Technology, IEEE Transactions on}, vol.~22, no.~2, pp. 413--427,
  2014.

\bibitem{tatikondacontrol}
S.~Tatikonda and S.~Mitter, ``Control under communication constraints,''
  \emph{Automatic Control, IEEE Transactions on}, vol.~49, no.~7, pp.
  1056--1068, 2004.

\bibitem{nair2007feedback}
B.~G.~N. Nair, F.~Fagnani, S.~Zampieri, and R.~J. Evans, ``Feedback control
  under data rate constraints: An overview,'' \emph{Proceedings of the IEEE},
  vol.~95, no.~1, pp. 108--137, 2007.

\bibitem{molin2009lqg}
A.~Molin and S.~Hirche, ``On {LQG} joint optimal scheduling and control under
  communication constraints,'' in \emph{Decision and Control, 2009 held jointly
  with the 2009 28th Chinese Control Conference. CDC/CCC 2009. Proceedings of
  the 48th IEEE Conference on}.\hskip 1em plus 0.5em minus 0.4em\relax IEEE,
  2009, pp. 5832--5838.

\bibitem{di2015co}
G.~Di~Girolamo, A.~D'Innocenzo, and M.~Di~Benedetto, ``Co-design of controller
  and routing redundancy over a wireless network,'' \emph{IFAC-PapersOnLine},
  vol.~48, no.~22, pp. 100--105, 2015.

\bibitem{bao2011iterative}
L.~Bao, M.~Skoglund, and K.~H. Johansson, ``Iterative encoder-controller design
  for feedback control over noisy channels,'' \emph{IEEE Transactions on
  Automatic Control}, vol.~56, no.~2, pp. 265--278, 2011.

\bibitem{ozekici1997markov}
S.~{\"O}zekici, ``Markov modulated bernoulli process,'' \emph{Mathematical
  Methods of Operations Research}, vol.~45, no.~3, pp. 311--324, 1997.

\bibitem{nevsic2004input}
D.~Ne{\v{s}}i{\'c} and A.~R. Teel, ``Input-output stability properties of
  networked control systems,'' \emph{Automatic Control, IEEE Transactions on},
  vol.~49, no.~10, pp. 1650--1667, 2004.

\bibitem{boyd2007tutorial}
S.~Boyd, S.-J. Kim, L.~Vandenberghe, and A.~Hassibi, ``A tutorial on geometric
  programming,'' \emph{Optimization and engineering}, vol.~8, no.~1, pp.
  67--127, 2007.

\bibitem{maranas1997global}
C.~D. Maranas and C.~A. Floudas, ``Global optimization in generalized geometric
  programming,'' \emph{Computers \& Chemical Engineering}, vol.~21, no.~4, pp.
  351--369, 1997.

\bibitem{zhang2006communication}
L.~Zhang and D.~Hristu-Varsakelis, ``Communication and control co-design for
  networked control systems,'' \emph{Automatica}, vol.~42, no.~6, pp. 953--958,
  2006.

\bibitem{kushner1967}
H.~Kushner, \emph{Stochastic stability and control}.\hskip 1em plus 0.5em minus
  0.4em\relax Academic Press, New York, 1967.

\bibitem{khasminskii2011stochastic}
R.~Khasminskii, \emph{Stochastic stability of differential equations}.\hskip
  1em plus 0.5em minus 0.4em\relax Springer Science \& Business Media, 2011,
  vol.~66.

\bibitem{altman1999constrained}
E.~Altman, \emph{Constrained Markov decision processes}.\hskip 1em plus 0.5em
  minus 0.4em\relax CRC Press, 1999, vol.~7.

\bibitem{semprog}
\BIBentryALTinterwordspacing
L.~Vandenberghe and S.~Boyd, ``Semidefinite programming,'' \emph{SIAM Review},
  vol.~38, no.~1, pp. 49--95, 1996. [Online]. Available:
  \url{http://dx.doi.org/10.1137/1038003}
\BIBentrySTDinterwordspacing

\bibitem{lewis2003robot}
F.~L. Lewis, D.~M. Dawson, and C.~T. Abdallah, \emph{Robot manipulator control:
  theory and practice}.\hskip 1em plus 0.5em minus 0.4em\relax CRC Press, 2003.

\bibitem{isidori1995nonlinear}
A.~Isidori, \emph{Nonlinear control systems}.\hskip 1em plus 0.5em minus
  0.4em\relax Springer Science \& Business Media, 1995.

\bibitem{costa2006discrete}
O.~L.~V. Costa, M.~D. Fragoso, and R.~P. Marques, \emph{Discrete-time Markov
  jump linear systems}.\hskip 1em plus 0.5em minus 0.4em\relax Springer Science
  \& Business Media, 2006.

\bibitem{jiang1994small}
Z.-P. Jiang, A.~R. Teel, and L.~Praly, ``Small-gain theorem for {ISS} systems
  and applications,'' \emph{Mathematics of Control, Signals and Systems},
  vol.~7, no.~2, pp. 95--120, 1994.

\bibitem{bonnans2013perturbation}
J.~F. Bonnans and A.~Shapiro, \emph{Perturbation analysis of optimization
  problems}.\hskip 1em plus 0.5em minus 0.4em\relax Springer Science \&
  Business Media, 2013.

\end{thebibliography}
	\begin{IEEEbiography}[]{Bin Hu}
received the B.S. degree in
automation from Hefei University of Technology,
Hefei, China, in 2007, the M.S. degree in control
and system engineering from Zhejiang University,
Hangzhou, China, in 2010, and the Ph.D. degree
in electrical engineering from the University of
Notre Dame, Notre Dame, IN, USA in 2016. His research interests include stochastic
networked control systems, information theory, switched control systems,
distributed control and optimization, and human machine interaction.
	\end{IEEEbiography}

	\begin{IEEEbiography}[]{Yebin Wang}
		received the B.Eng. degree in
Mechatronics Engineering from Zhejiang University,
China, in 1997, M.Eng. degree in Control Theory
\& Engineering from Tsinghua University, China,
in 2001, and Ph.D. in Electrical Engineering from
the University of Alberta, Canada, in 2008. Dr.
Wang has been with Mitsubishi Electric Research
Laboratories in Cambridge, MA, USA, since 2009,
and now is a Senior Principal Research Scientist.
From 2001 to 2003 he was a Software Engineer,
Project Manager, and R\&D Manager in industries,
Beijing, China. His research interests include nonlinear control and estimation,
optimal control, adaptive systems and their applications including mechatronic
systems.
	\end{IEEEbiography}

	\begin{IEEEbiography}[]{Philips Orlik}
		was born in New York, NY in 1972.
He received the B.E. degree in 1994 and the M.S.
degree in 1997 both from the State University of
New York (SUNY) at Stony Brook. In 1999 he
earned his Ph. D. in electrical engineering also from
SUNY Stony Brook. He is currently the Group
Manager of Electronics \& Communications at Mitsubishi
Electric Research Laboratories Inc. located
in Cambridge, MA. His primary research focus is
on advanced wireless and mobile cellular communications,
sensor networks, ad-hoc networking and
UWB. Other research interests include vehicular/car-to-car communications,
mobility modeling, performance analysis, and queuing theory.
	\end{IEEEbiography}

	\begin{IEEEbiography}[]{Toshiaki Koike-Akino}
(M'05-SM'11)	received the
B.S. degree in electrical and electronics engineering,
M.S. and Ph.D. degrees in communications and
computer engineering from Kyoto University, Kyoto,
Japan, in 2002, 2003, and 2005, respectively. During
2006–2010, he has been a Postdoctoral Researcher
at Harvard University, and joined Mitsubishi Electric
Research Laboratories, Cambridge, MA, USA, since
2010. His research interest includes digital signal
processing for data communications and sensing.
He received the YRP Encouragement Award 2005,
the 21st TELECOM System Technology Award, the 2008 Ericsson Young
Scientist Award, the IEEE GLOBECOM'08 Best Paper Award in Wireless
Communications Symposium, the 24th TELECOM System Technology Encouragement
Award, and the IEEE GLOBECOM'09 Best Paper Award in
Wireless Communications Symposium.
\end{IEEEbiography}

	\begin{IEEEbiography}[]{Jianlin Guo}
	is a Senior Principal Research Scientist
at Mitsubishi Electric Research Laboratories
in Cambridge, Massachusetts, USA. He received
his Ph.D. in Applied Mathematics in 1995 from
University of Windsor, Windsor, Ontario, Canada.
His research interests include routing and resource
management in wireless IoT networks, coexistence
of the heterogeneous wireless networks, control over
wireless networks, wireless sensor networks, smart
grid networks, safety and handover in vehicular
communications, nonlinear stability of convection in
porous medium.
\end{IEEEbiography}
		
\end{document}